\documentclass[11pt]{article}
\usepackage[latin1]{inputenc}
\usepackage{float}
\usepackage{multirow}
\usepackage{color}
\usepackage{amsmath,amssymb,amsthm,amsfonts,amstext,amsbsy,amscd}
\usepackage{mathrsfs}
\usepackage{array}
\usepackage{rotating}
\usepackage{enumitem}
 \usepackage{amssymb,amsmath,comment,graphicx}
\usepackage{multicol}
\usepackage{bbm}
\usepackage{faktor}
\usepackage{tikz,tkz-euclide}
\usepackage{epstopdf}
\usepackage{dsfont}
\usepackage{mwe} 
\usepackage[acronym]{glossaries}
\RequirePackage[authoryear]{natbib}
\usepackage{subcaption}

\newcommand{\rmd}{{\mathrm{d}}}
\newcommand{\I}[1]{\mathbbm{1}_{\{#1\}}}
\newcommand{\R}{\mathbb{R}}
\newcommand{\N}{\mathbb{N}}
\newcommand{\PP}{\mathbb{P}}
\newcommand{\V}{\mathbb{V}}
\newcommand{\corr}{\mathrm{corr}}
\newcommand{\Z}{\mathbb{Z}}
\renewcommand{\P}{\mathbb{P}}
\newcommand{\E}{\operatorname{\mathbb{E}}}
\newcommand{\Cov}{\mathrm{Cov}}

\newcommand{\ii}{{\mathbf i}}
\newcommand{\uu}{{\mathbf u}}

\newcommand{\jj}{{\mathbf j}}

\newcommand{\1}{{\mathbf 1}}
\newcommand{\be}{{\mathbf e}}
\newcommand{\vect}{{\rm vect}}
\newcommand{\0}{{\mathbf 0}}
\renewcommand{\hat}{\widehat}
\renewcommand{\bar}{\overline}
\newcommand{\eps}{\varepsilon}
\newcommand{\as}{\qquad\mathrm{\textit{a.s.}}}

\theoremstyle{plain}
\newtheorem{Theorem}{Theorem}[section]
\newtheorem{Corollary}{Corollary}[section]
\newtheorem{Lemma}{Lemma}[section]
\newtheorem{Proposition}{Proposition}[section]
\theoremstyle{definition}
\newtheorem{Definition}{Definition}[section]
\newtheorem{Remark}{Remark}
\newtheorem{Example}{Example}
\renewcommand{\hline}{------------------------------------------------------}

\usepackage{lscape} 
\usepackage{ifpdf}
\ifpdf
\else
\usepackage{graphicx}
\DeclareGraphicsExtensions{.eps,.jpg,.png}
\fi
\usepackage{epstopdf}

\usepackage[]{algorithmic}
\usepackage[0]{algorithm}

\floatname{algorithm}{}
\definecolor{gris}{gray}{0.5} 

\usepackage[bookmarks = true, colorlinks=true, linkcolor = blue, citecolor = blue, menucolor = black, urlcolor = black, plainpages=false]{hyperref}

\usepackage[left=2.1cm, top=2.1cm, right=2.1cm, nohead, foot=0.6cm]{geometry}

\begin{document} 
\title{Surface area and volume of excursion sets observed on  point cloud based polytopic tessellations}

\author{\sc Ryan Cotsakis\footnote{Universit\'{e} C\^{o}te d'Azur, CNRS  UMR  7351,   Laboratoire J.A. Dieudonn\'{e} (LJAD),   France},\,   Elena Di Bernardino$^{*}$  and C\'{e}line Duval\footnote{Universit\'e de Lille, CNRS  UMR 8524, Laboratoire Paul Painlev\'{e},  Lille, France}}
\date{}
\maketitle
\begin{abstract}
The excursion set of a $C^2$ smooth random field carries relevant information in its various geometric measures.
From a computational viewpoint, one never has access to the continuous observation of the excursion set, but rather to observations at discrete points in space.
It has been reported that for specific regular lattices of points in dimensions 2 and 3, the usual estimate of the surface area of the excursions remains biased even when the lattice becomes dense in the domain of observation.
In the present work, under the key assumptions of stationarity and isotropy, we demonstrate that this limiting bias is invariant to the locations of the observation points. Indeed, we identify an explicit formula for the bias, showing that it only depends on the spatial dimension $d$.
This enables us to define an unbiased estimator for the surface area of excursion sets that are approximated by general tessellations of polytopes in $\R^d$, including Poisson-Voronoi tessellations.
We also establish a joint  central limit theorem for the surface area and volume estimates of excursion sets observed over hypercubic lattices.

\end{abstract}



\section{Introduction}
 
\subsection{Motivations}

The study of random fields through the geometry of their excursion sets has received a lot of interest in recent literature. This is mainly stimulated by their wide range of applications in domains such as cosmology, for  the study of Cosmic Microwave Background radiation and the distribution of galaxies (see, \emph{e.g.}, \citet{Planck},   \citet{SG}, \citet{GCPPM},  \citet{GHVK}), brain imaging (see \citet{AT11}, Section 5, and the references therein), and the modelling of sea waves (see, \emph{e.g.}, \citet{LH57},   \citet{Wsc85},  \citet{Lin00}).
The geometric features considered are referred to as either Lipschitz-Killing curvatures, intrinsic volumes or Minkowski functionals, depending on the literature. Many studies have been dedicated to computing these objects from the observation of one excursion set on a compact domain $T$ in $\R^d$ (see \textit{e.g.} \cite{AT07}), limit results when the size of $T$ grows to $\R^{d}$ have been established under specific conditions on the field (see \textit{e.g.} \cite{bulinski2010central}, \cite{bulinski2012central}, \cite{Kratz2018}, \cite{meschenmoser2013functional} or \cite{Spodarev13}) and these estimates have been successfully used to derive testing procedures (see \textit{e.g.} \cite{abaach:hal-03138682}, \cite{bierme2018lipschitz}, \cite{di2022statistics} or \cite{Berzin}).
In the present work, we focus on two specific additive measures of the excursion sets of a $C^{2}$ stationary  random field in $\R^{d}$: the surface area and the volume (see Section~\ref{sec:defsig} for their mathematical definition).

The main motivation of this work comes from the following ascertainment: many works rely on the assumption that ``\emph{the excursion set is observed on $T\subset\R^{d}$}'' which is meant to be understood as ``the field is \emph{continuously} observed over $T$''. This seems unrealistic as in practice, for instance in dimension 2 where excursion sets can be viewed as images,  they are encoded through a matrix whose entries make a one-to-one connection with the pixels of the image. Even in cases where the resolution of the image is very high, the image of the excursion set remains a discretization of the continuous object that is the excursion set on $T$.

 There exist estimators of the surface area and the volume taking into account the discrete nature of the observations. For the surface area, whose computation is more challenging, local counting algorithms are studied for example, in \cite{miller1999} and \cite{bierme:hal-02793752}.  Both studies consider specific two-dimensional regular lattices of points (square and hexagonal) and bring to light an interesting phenomenon: for each of the considered lattices, the expected surface area (in this case, perimeter length) of the discretized excursion set does not converge to that of the continuous excursion set; there is a bias factor of $4/\pi$ (see \cite{bierme:hal-02793752}, Proposition 5). In dimension 3, \cite{miller1999} observes a similar behavior, with a bias factor of $3/2$ for the cubic lattice. Interestingly, \cite{thale2016asymptotic} shows that in any dimension $d$, for an excursion set  observed on a Poisson-Voronoi mosaic, ``\emph{the surface area asymptotics involve a universal correction factor}''. However, this correction factor is not explicitly calculated in the aforementioned paper. In dimension 2, for a square lattice, a statistical estimating strategy  is developed in \cite{cotsakis2022perimeter} to avoid this asymptotic bias (see \cite{cotsakis2022perimeter}, Remark 3 for further details)  at the cost of an introduced hyperparameter.

The main result of this paper offers a general picture for the mean surface area of an excursion set that is approximated by convex polytopes in the following sense: for a general tessellation of polytopes in $\R^d$ (in the sense of Definition~\ref{def:lattice}), we provide an explicit formula for the limiting mean surface area of the polytopic region that approximates an excursion set as the scale of the tessellation is decreased. The limiting mean surface area of the approximated excursion set approaches a constant multiple of that of the true  excursion set; surprisingly, the constant is independent of the geometry of the tessellation, and only depends on the dimension $d$ (see Equation~\eqref{eqn:honeycomb_deterministic}). We compute the constant in every dimension (see Equations~\eqref{eqn:honeycomb_deterministic} and~\eqref{betad}) and identify it with the unspecified \emph{universal correction factor} mentioned in \cite{thale2016asymptotic} (see our Corollary~\ref{corollaryBiais}). 

Moreover, thanks to a second order expansion of this bias (see Theorem~\ref{prp:Pcrossing_Esurfacearea}), it is possible, to derive for a \emph{hypercubic} lattice a \emph{joint} central  limit theorem for the estimated surface area and volume (see Theorem~\ref{CLTjoint}) by imposing additional strong mixing assumptions on the  underlying  field. This limit result is novel among existing limit results since it gives the joint asymptotics of two \emph{different} Lipschitz-Killing curvatures, the surface area and the volume, whereas most multivariate limit theorems hold for a single curvature at multiple levels (see for instance Corollary~\ref{CLT}).\\

The outline of the paper is the following. In Section~\ref{sec:Crof} we define the geometric measures that we consider, as well as the Crofton formula, which is an essential tool to prove the main result (Theorem~\ref{prp:Pcrossing_Esurfacearea}). In Section~\ref{sec:defsig}, the surface area and volume as well as their corresponding estimators on general point clouds are introduced. Since the bias of the estimated {volume} is a well understood deterministic quantity that is asymptotically negligible, Section~\ref{sec:BiasSA} focuses on the study of the bias of the estimated {surface area}; the main results, which hold for general point clouds in any dimension (Theorems~\ref{prp:Pcrossing_Esurfacearea} and~\ref{thm:honeycomb}), are stated and proved. Section~\ref{sec:TCL} restricts to the hypercubic lattice and proposes under additional strongly-mixing assumptions the joint CLT (Theorem~\ref{CLTjoint}) for the estimated surface area and volume. Sections~\ref{prf:secSA} and~\ref{prf:secTCL} contain additional results and proofs related to Sections~\ref{sec:BiasSA} and~\ref{sec:TCL} respectively.  Finally, an Appendix Section includes a refined result for the variance of the estimated volume (see Section~\ref{QuasiAssociativitySection}), several examples (see Section~\ref{sec:annexeExamples}) and some considerations on alternative approaches to recover the dimensional constant  appearing in Theorem~\ref{prp:Pcrossing_Esurfacearea} and to approximate   the  surface area  (see Section~\ref{DimensionalConstantSQUARE}).

\subsection{Geometric measures and the Crofton formula\label{sec:Crof}}

In the following,   $\| \cdot \|_p$  denotes the $L_p$  norm; $\| \cdot \|_{\infty}$, the supremum norm;   $| \cdot |$,  the absolute value; $\mathbbm{1}_{A}$, the indicator of a set~$A$; and $\partial A$, the boundary of a set $A$. The closed ball of radius $r$ centered at the origin $\0$ in $\R^d$ is denoted $B^d_r$. Finally, recall that $(\be_{i})_{1\le i\le d}$ denotes the canonical basis of $\R^{d}$.

\paragraph{Hausdorff measures.}

We first introduce the different measures considered in this article.  For $k\in\{0,\ldots,d\}$, let $\sigma_k(B)$ be the $k$-dimensional Hausdorff measure of a measurable set $B \subset \R^d$,
\begin{align} \label{measures}
\sigma_{k}(B):= \frac{\pi^{k/2}}{2^{k}\Gamma\left(\frac k2 + 1\right)}\lim_{\delta \rightarrow 0} \sigma_{k}^{\delta}(B),
\end{align}
where $\Gamma$ denotes the gamma function and
\begin{align} \label{deltacovering}
\sigma_{k}^{\delta}(B) := \inf\left\{\sum_{i\in \N} \mbox{diam}(U_i)^{k} : \mbox{diam}(U_i) < \delta,\  \bigcup_{i=1}^\infty U_i \supseteq B \right\},
\end{align}
where diam$(U) :=\sup\{\|u-v\|_{2}, u,v\in U\}$ and the infimum is taken over all countable covers of $B$ by arbitrary subsets $U_{i}$ of $\R^{d}$ (see, \emph{e.g.}, \cite{Weil08} p.634). The Hausdorff dimension of $B$ is the unique integer value $ d_{B}$ such that $\sigma_{k}(B)=0$ if $k<d_{B}$ and $\sigma_{k}(B)=+\infty$ if $k>d_{B}$ (see, \emph{e.g.}, \cite{rogers1998hausdorff}).
We have chosen to normalize  $\sigma_{k}(B)$ in \eqref{measures}  such that for $k=\{0, \ldots, d\}$, $\sigma_k(B)$ corresponds to the $k$-dimensional Lebesgue measure of $B$.

In the present work we focus on the measures in~\eqref{measures} for   $k=d-1$ and $k= d$. They correspond respectively to what we call \textit{surface area} and \textit{volume} in dimension $d$. In the proofs, two other measures play an important role: $\sigma_{0}$, the counting measure for sets of isolated points,  and  $\sigma_{1}$, the measure of length.

\paragraph{Main tool: Crofton formula.}
In the present work,  we heavily  rely  on the Crofton formula. For a $k$-dimensional smooth object embedded in $\R^d$, this classical result in integral geometry relates the $\sigma_{k}$ measure of the object with the average number of times it is intersected by randomly oriented $(d-k)$-flats. We will be particularly interested in the case of $k=d-1$, in which the $1$-flats correspond to randomly oriented lines.\medskip

To  state the Crofton formula, we first need to introduce the affine Grassmanian $A(d,m)$, for $m\in\{1,...,d\}$, which is the set of affine $m$-dimensional subspaces of $\R^d$.  Since the set $A(d,1)$ of lines in $\R^{d}$ plays a crucial role in the Crofton formula  we  use, we propose a particular parametrization.

It is shown in \cite{Weil08} p.168 that $A(d,1)$ is equipped with a unique locally finite motion invariant measure $\mu_1$, that is normalized such that
\begin{equation*}
\mu_1(\{l\in A(d,1): l\cap B^d_1\neq \emptyset\}) = \sigma_{d-1}(B^{d-1}_1).
\end{equation*}
For $\mathbf{s}\in \partial B_1^d$ and $\mathbf{v}\in\vect(\mathbf{s}^\perp)$, denote by $l_{\mathbf{s},\mathbf{v}}$ the element of $A(d,1)$ that is parallel with $\mathbf{s}$ and passes through the point $\mathbf{v}$,
\begin{equation}\label{eqlines}
l_{\mathbf{s},\mathbf{v}}:=\{\mathbf{v}+\lambda\mathbf{s}:\lambda\in\R\}.
\end{equation}   For each $\mathbf{s}\in \partial B_1^d$, there exists a unitary (rotation) operator $\theta_\mathbf{s}$ that maps $\be_1$ to $\mathbf{s}$. Therefore, we define the parametrization $\varphi:\partial B^d_1 \times \R^{d-1}\rightarrow A(d,1)$ satisfying
$\varphi(\mathbf{s},\mathbf{u})=l_{\mathbf{s}, \mathbf{v_\mathbf{s}(u)}}$, with $\mathbf{v_\mathbf{s}(u)}:=\theta_\mathbf{s}\circ\begin{pmatrix}
0\\
\mathbf{u}
\end{pmatrix}\in\vect(\mathbf{s}^\perp)$.
Notice that for $E\subset A(d,1)$,
$$\mu_1(E) = \frac{(\sigma_{d-1}\otimes \sigma_{d-1})(\varphi^{-1}(E))}{\sigma_{d-1}(\partial B^d_1)},$$
where $\otimes$ denotes the product measure.\\

We recall here a particular version of the Crofton formula in Theorem~5.4.3 in \cite{Weil08}, which implies that for a manifold $M\subset \R^d$ satisfying $0 < \sigma_{d-1}(M) < \infty$, it holds that
\begin{equation}\label{eqn:crofton_appendix}
\sigma_{d-1}(M) = \frac{\sqrt{\pi}\ \Gamma(\frac{d+1}{2})}{\Gamma(\frac{d}{2})}
\int_{A(d,1)} \sigma_0(M\cap l)\ \mu_1(\rmd l).
\end{equation}
By writing $A(d,1)$ in terms of the above parametrization $\varphi$, Equation~\eqref{eqn:crofton_appendix} takes the form
\begin{equation}\label{eqn:crofton}
\sigma_{d-1}(M) = \frac{\sqrt{\pi}\ \Gamma(\frac{d+1}{2})}{\Gamma(\frac{d}{2})}
\int_{\R^{d-1}}\int_{\partial B^d_1} \frac{\sigma_0(M\cap l_{\mathbf{s},\mathbf{v_\mathbf{s}(u)}})}{\sigma_{d-1}(\partial B^d_1)}\ \rmd \mathbf{s}\ \rmd\mathbf{u}.
\end{equation}
A helpful interpretation of the Crofton formula in Equation~\eqref{eqn:crofton} is as follows: the expected $\sigma_{d-1}$ measure of the projection of $M$ on a $(d-1)$-dimensional hyperplane with uniformly random orientation is simply a constant  multiple   of $\sigma_{d-1}(M)$.\\

The Crofton formula can also be exploited to propose algorithms to compute the surface area, \textit{e.g.}, it is considered in \cite{lehmann2012efficient} for  objects in 2 and 3 dimension and  recently in  \cite{aaron2020surface} which provides consistent estimators for the surface area of a compact domain $S$ from the observation  of \emph{i.i.d.} random variables supported on $S$. The interested reader is referred to Appendix~\ref{sec:annexeExamples} for an illustration of Equation~\eqref{eqn:crofton} on two simple examples.

\subsection{Estimated volume and surface area of excursion sets observed over a point cloud\label{sec:defsig}}

\paragraph{Excursion sets and associated measures.} We now apply the previous $\sigma_{d}$ and $\sigma_{d-1}$ measures to specific manifolds: the excursion sets of $d$-dimensional smooth random fields.

\begin{Definition}[$\sigma_{d}$ and $\sigma_{d-1}$ measures of excursion sets] \label{defLK}
Let  $\{X(t), t \in \R^d\}$,  for  $d \geq 2$,  be a random field   satisfying the following assumption
\begin{enumerate}[label={($\mathcal{A}$0)}]
\item\label{A0}
$X$ is stationary with  positive finite variance and is almost surely twice differentiable. Furthermore, the probability density of $\big(X(\0),\nabla X(\0)\big)$ is bounded uniformly on $\R^{d+1}$.
\end{enumerate}
Let   $u \in \R$ and    $T \subset \R^d$  be a bounded  closed  hypercube with non empty interior.
 We consider the excursion set within $T$ above level $u$:
$$E^{T}_{X}(u):= \{t\in T\,:\,X(t)\ge u\}= T \cap E_{X}(u),\quad \mbox{where }E_{X}(u):= X^{-1}([u,+\infty)).$$
Similarly, the level curves within $T$ are defined by
\begin{center}
$L_X^T(u) := \{t\in T\,: \,X(t) = u\} = T \cap \partial E_{X}(u)$, \emph{a.s.}
\end{center}
Remark that Assumption~\ref{A0} guarantees that $X$ admits no critical points at the level $u$ almost surely, which implies that $L_X^T(u)$ is a $(d-1)$-dimensional manifold possessing a $\sigma_{d-1}$ measure with finite first and second moments (see, \textit{e.g.}, \cite{cabana1987affine}, Theorem 11.2.1 and Lemma 11.2.11 in \cite{AT07}).   In addition, since the considered random field $X$ is of class $C^2$ \textit{a.s.}, the random set $E_X(u)$ is a $C^2$ submanifold of $\R^d$ and its intersection with the compact, convex hypercube $T$ provides the positive reach property (see \cite{bierme2018lipschitz}). \\

Define the normalized $\sigma_{d}$ and $\sigma_{d-1}$ measures of the excursion set, for $u\in\R$, as
\begin{align}
C^T_{d-1}(u) : &=   \frac{1}{\sigma_d(T)} \sigma_{d-1}\big(L_X^T(u)\big) = \frac{1}{\sigma_d(T)}\int_{L^{T}_X(u)}  \sigma_{d-1}(\rmd s),  \label{eq:UestCT1bis}\\
C^T_d(u)  :& =  \frac{1}{\sigma_d(T)} \sigma_d\big(E_X^T(u)\big)= \frac{1}{\sigma_d(T)}\int_{T} \mathbbm{1}_{\{X(t)\ge u\}}	\rmd t.\label{eq:UestCT2}
 \end{align}
Assumption \ref{A0} guarantees the existence of the associated  densities
\begin{align} \label{defCast}
C_k^{*}(u) := \E[C^T_k(u)],~\mbox{ for } k=d, d-1,
\end{align}
which are independent of the size of the hypercube $T$. \end{Definition}
The independence of $C_d^{*}$ from the size of $T$ is trivially verified  using that  $X$ is stationary and  Fubini-Tonelli theorem which give immediately that the density of the normalised volume satisfies
\begin{equation} \label{areadensity}
C_d^{*}(u) = \frac{1}{\sigma_d(T)}\E\Big[\int_{T} \mathbbm{1}_{\{X(t)\ge u\}}	\rmd t\Big] = \PP(X(\0) \geq u).
\end{equation}
Furthermore, note that we consider in~\eqref{eq:UestCT1bis} the $\sigma_{d-1}$ measure of $L_{X}^{T}(u)=T \cap \partial E_{X}(u)$ and not of $\partial E^{T}_{X}(u)$. Therefore, from Definition  2.1  and Proposition 2.5  in  \cite{bierme2018lipschitz}, we get via  kinematic formulas that  $\E[C^T_{d-1}(u)]$ is equal to the surface area density. Indeed we do not add the artificial contribution of $\partial T$ to the level curves in Definition~\ref{defLK}. Notice that the density $C^*_{d-1}(u)$ can be   explicitly obtained for certain specific random fields. Two classical examples (the isotropic Gaussian  and chi-square  random  fields) are presented in Appendix~\ref{sec:annexeExamples} (see Example \ref{ModelsGaussianType1}). \\

The random quantities in~\eqref{eq:UestCT1bis}-\eqref{eq:UestCT2}  can only be used as estimators of $C^{*}_{d}(u)$ and $C^{*}_{d-1}(u)$ if we observe  the excursion set $E_{X}^{T}(u)$ on the whole domain $T$. In practice, or at least numerically,  images of excursion sets are not objects defined on all $T$ but discretely encrypted objects, \textit{i.e.}, for each point of a discrete  grid. Then,    quantities in~\eqref{eq:UestCT1bis}-\eqref{eq:UestCT2}  are never empirically accessible. In the remainder of this section we propose estimators of $C^{*}_{d}(u)$ and $C^{*}_{d-1}(u)$ from the observation of the excursion set on a general point cloud (\emph{i.e.}, based on the knowledge of which points fall in the excursion set).

\paragraph{Polytopic tessellations based on point clouds.}

For an arbitrary point cloud, we describe the set of tesselations of $\R^d$ that are permissible for the construction of our estimators of $C^{*}_{d}(u)$ and $C^{*}_{d-1}(u)$. 

\begin{Definition}\label{def:lattice}
Let $\mathcal{H}$ be a set of convex, closed polytopes that tessellates $\R^d$ in such a way that satisfies the following condition. To each $P\in \mathcal{H}$, one can assign a reference point $P^{\bullet}\in P$ such that for any two adjacent cells $P_1,P_2\in\mathcal{H}$, the intersection of their boundaries  is normal to the vector spanned between $P_1^\bullet$ and $P_2^\bullet$. We say that $\mathcal{H}$ is \textit{point-referenceable}, and that the set of pairs $\dot{\mathcal{H}} = \{(P,P^\bullet):P\in\mathcal{H}\}$ is a \textit{point-referenced $d$-honeycomb}. Let $\mathfrak{T}^d$ denote the space of point-referenceable $d$-honeycombs, so that $\mathcal{H}\in\mathfrak{T}^d$.
\end{Definition}

Recall that the interiors of the polytopes in a tessellation do not intersect, \emph{i.e.}, $\sigma_{d}(P_{1}\cap P_{2})=0$, for any different $P_1,P_2\in\mathcal{H}$. However, if two polytopes $P_1$ and $P_2$ are \textit{adjacent} in $\mathcal{H}$, then $\sigma_{d-1}(P_1\cap P_2)>0$.

\begin{Remark}
The set of point-referenceable $d$-honeycombs $\mathfrak{T}^d$ in Definition~\ref{def:lattice} contains a wide variety of  tessellations. For example, the Voronoi diagram of any point cloud in $\R^d$ is in $\mathfrak{T}^d$. It contains also $d$-honeycombs that cannot be realised as a Voronoi diagram as, for example,  the Archimedean tessellations   or the  tiling of $\R^2$ in Figure~\ref{fig:lattice} (a). Simple examples of tessellations of $\R^2$ that are not in $\mathfrak{T}^2$ are Pythagorean tilings and, of course, tilings with curved tiles.
\end{Remark}

Figure~\ref{fig:lattice} provides some examples of point-referenceable tessellations of $\R^2$. Notice that regular triangular and regular hexagonal lattices can be seen as limiting cases of the tessellation in Figure~\ref{fig:lattice}   (a); furthermore, as depicted, it is not a Voronoi tessellation. Figure~\ref{fig:lattice} (b) represents the well-known  square lattice, and panel (c) depicts the Voronoi tessellation of an arbitrary point cloud.

\begin{figure}
[H]
  \centering
  \hspace{0.5cm}
  \begin{subfigure}{0.25\textwidth}
    \includegraphics[width=\linewidth]{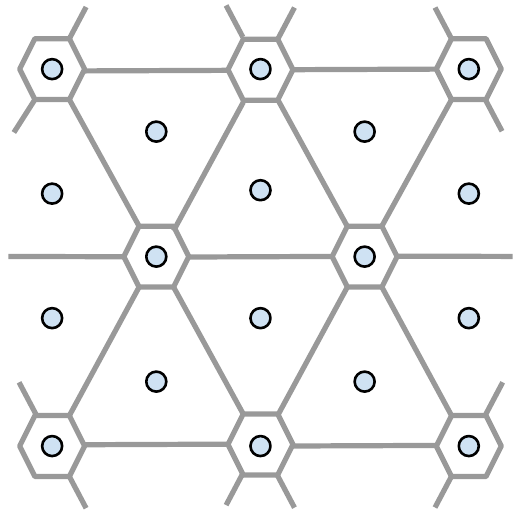}
    \caption{Hexagones and     truncated triangles lattice} \label{fig:1a}
  \end{subfigure}
  \hspace*{\fill}    
  \begin{subfigure}{0.23\textwidth}
    \includegraphics[width=\linewidth]{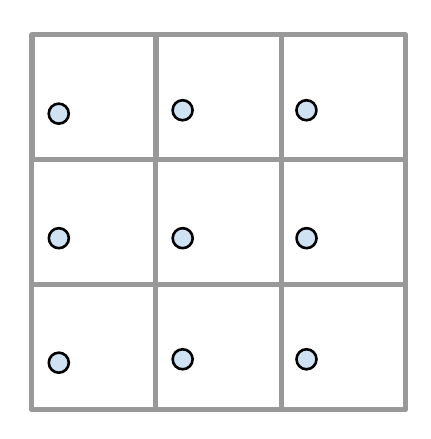}
    \caption{Square lattice} \label{fig:1b}
  \end{subfigure}
  \hspace*{\fill}    
  \begin{subfigure}{0.285\textwidth}
    \includegraphics[width=\linewidth]{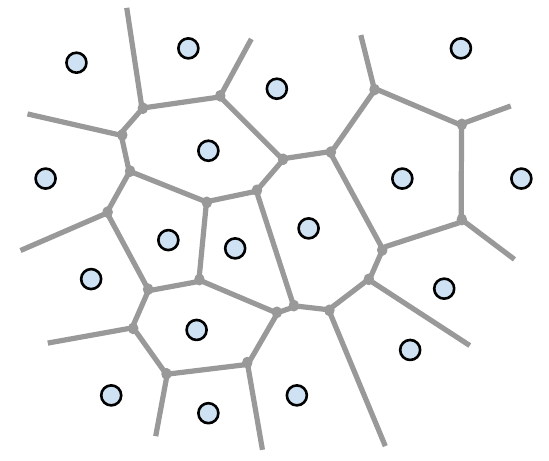}
    \caption{Voronoi tessellation} \label{fig:1c}
  \end{subfigure}\hspace{0.5cm}

\caption{Three examples of tilings in $\mathfrak{T}^2$ (see Definition~\ref{def:lattice})  with a particular choice of reference points shown as light-blue circles. The resulting structures are  point-referenced 2-honeycombs.}\label{fig:lattice}
\end{figure}

\paragraph{Estimators for $C_{d}^{T}(u)$ and $C^{T}_{d-1}(u)$.}
For a given point cloud in $\R^d$, that can be random, it is always possible to construct a point-referenced $d$-honeycomb (see Definition~\ref{def:lattice}) with the point cloud as its reference points (take the Voronoi diagram, for example). Thus, for each element of $\mathfrak{T}^d$ there is at least one corresponding point cloud, and to each point cloud in $\R^d$ there is at least one corresponding element of $\mathfrak{T}^d$.
We propose the following estimators for the quantities in~\eqref{eq:UestCT1bis} and~\eqref{eq:UestCT2} based on the knowledge of at which points in the point cloud the random field $X$ exceeds the level $u$.

\begin{Definition}\label{def:estGenTil} Let $T\subset \R^d$ be a compact domain  with non empty interior.  Let $\mathcal{H}\in \mathfrak{T}^d$, and let $\dot{\mathcal{H}} = \{(P,P^\bullet):P\in\mathcal{H}\}$ be a corresponding point-referenced $d$-honeycomb in the sense of Definition~\ref{def:lattice}.
Let $\mathcal{H}^T\subset \mathcal{H}$ be the set of polytopes $P$ in $\mathcal{H}$ such that $P\subseteq T$. We define  an estimator of $C_{d-1}^{T}(u)$   in~\eqref{eq:UestCT1bis} as
\begin{align}\label{GeneralLatticeEstimaor}
\widehat C^{(\dot{\mathcal{H}},T)}_{d-1}(u) &:= \frac{1}{\sigma_d(T)}\underset{P_1\neq P_2}{\sum_{P_1,P_2\in\mathcal{H}^T}}\sigma_{d-1}(P_1\cap P_2)\I{X(P_1^\bullet) \leq u < X(P_2^\bullet)}
\end{align}
and  of $C_{d}^{T}(u)$  in~\eqref{eq:UestCT2} as
\begin{equation}\label{eqCdHoneycomb}
\widehat C^{(\dot{\mathcal{H}},T)}_{d}(u) := \frac{1}{\sigma_d(T)}\sum_{P\in \mathcal{H}^T}\sigma_{d}(P)\I{X(P^\bullet) \geq u}.
\end{equation}
\end{Definition}

Computing the quantities in Definition~\ref{def:estGenTil} only requires the knowledge of the excursion set $E^{T}_{X}(u)$ on the reference points of a point-referenced $d$-honeycomb \textit{i.e.} a black and white image indicating at which points $P^{\bullet}$, for $P\in \mathcal{H}^{T}$, the field is above level $u$. In Equation~\eqref{GeneralLatticeEstimaor}, since the role of $P_{1}$ and $P_{2}$ is symmetric, both $\I{X(P_1^\bullet) \leq  u < X(P_2^\bullet)}$ and $\I{X(P_2^\bullet) \leq  u < X(P_1^\bullet)}$ are evaluated in the sum. Notice that it is crucial that the cells $P\in\mathcal{H}$ are closed to ensure that for adjacent cells $\sigma_{d-1}(P_{1}\cap P_{2})>0.$  Figure~\ref{fig:blue_paths} provides an illustration of the behaviour of the estimator in~\eqref{GeneralLatticeEstimaor} for the point-referenced 2-honeycombs in Figure \ref{fig:lattice}.

\begin{figure}[ht!]
    \centering
    \includegraphics[width=0.8\textwidth]{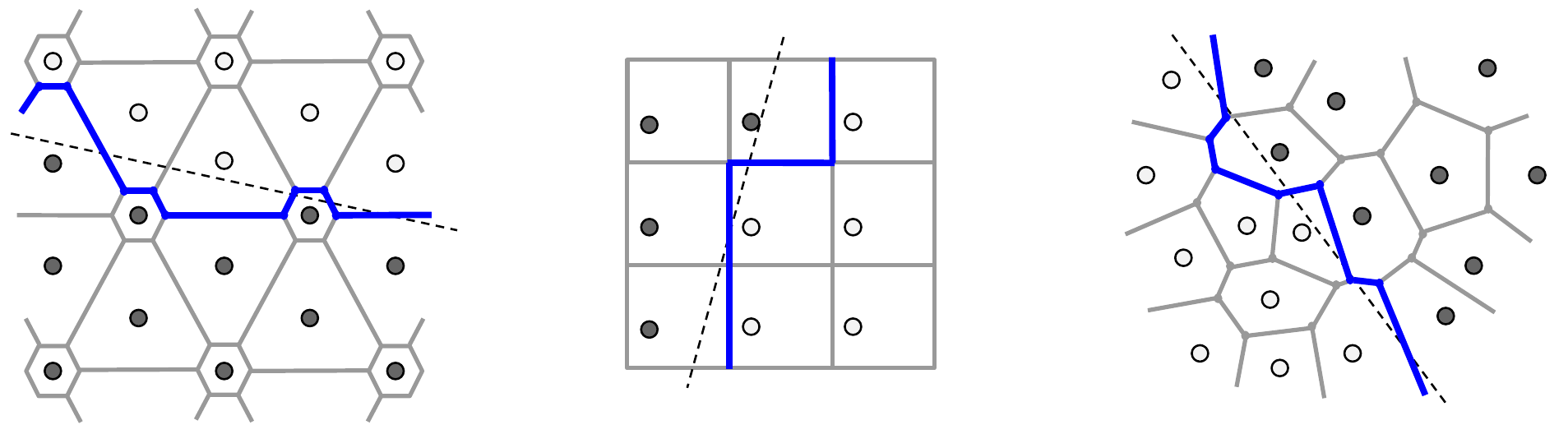}
    \caption{ \label{fig:blue_paths}
    For each of the point-referenced 2-honeycombs in Figure~\ref{fig:lattice}, a portion of the  level curve $\{X=u\}$ is illustrated as a dotted line. The estimator $\widehat C^{(\dot{\mathcal{H}},T)}_{d-1}(u)$  in~\eqref{GeneralLatticeEstimaor}  corresponds to the length of the resulting blue curve that separates the reference points in the same way as the curve $\{X=u\}$.}
\end{figure}

It is easy to see that  $\widehat C^{(\dot{\mathcal{H}},T)}_{d}(u)$ in~\eqref{eqCdHoneycomb}  is, in general, biased. Indeed, it follows from the stationarity in Assumption~\ref{A0} that
\begin{align}\label{eqbiaisarea}
\E[\widehat C^{(\dot{\mathcal{H}},T)}_{d}(u)]=\frac{\sigma_d\Big(\bigcup_{P\in\mathcal{H}^T}P\Big)}{\sigma_{d}(T)}\,C^{*}_{d}(u).
\end{align} Nevertheless, the bias factor ${\sigma_d\big(\bigcup_{P\in\mathcal{H}^T}P\big)}/{\sigma_{d}(T)}\leq 1$ is deterministic and it only relies on the structure of the point-referenced $d$-honeycomb. It is nearly unity if
$\sup\{\mathrm{diam}(P):P\in\mathcal{H}\} \ll \mathrm{diam}(T)$ and it is exactly unity for tessellations that satisfy $\bigcup_{P\in\mathcal{H}^T}P = T$. This can be easily obtained for simple tessellations such as the hypercubic one,  specifically described in Section~\ref{hypercubic:section}.
This is why we focus our attention to the bias  in~\eqref{GeneralLatticeEstimaor} for the $(d-1)$-volume of $L_X^T(u)$. In this case, the bias does not approach unity as the tile sizes become negligibly small with regards to the size of $T$, as previously observed \textit{e.g.} in \cite{bierme:hal-02793752} and \cite{miller1999}  for some periodic tessellations and in \cite{thale2016asymptotic} for Poisson-Voronoi mosaics. Explicit calculations of the bias have been made for square and hexagonal tilings in dimension 2 (\cite{bierme:hal-02793752}) and cubic lattices in dimension 3 (\cite{miller1999}). The approach in \cite{miller1999} is adapted to the $d$-dimensional hypercubic setting in Appendix \ref{DimensionalConstantSQUARE} (see Equation \eqref{constantbetad}).

Our Theorem~\ref{thm:honeycomb} provides a general formula for this limiting bias that holds in arbitrary dimension $d$, for a large subset of point-referenced $d$-honeycombs.

\section{Bias of the estimated surface area\label{sec:BiasSA}}

\paragraph{Assumptions.}

To control the bias of the estimated surface area we require additional sufficient assumptions   related to the regularity of the derivatives of $X$ in univariate directions.
\begin{enumerate}[label={($\mathcal{A}$1)}]
\item\label{A1}  Fix $\mathbf{w}\in\partial B_1^d$.
Assume   that $Y:=\{X(t \mathbf{w}), \, t \in \R\}$   is such that, $Y(0)\big|\{Y'(0)=0\}$ has a bounded density  and $\|Y''\|_{\infty,[0,1]} := \sup_{t\in[0,1]}|Y''(t)|\in L^2$.
\end{enumerate}

\begin{enumerate}[label={($\mathcal{A}$2)}]
\item\label{A2}
Let $L_X(u) := X^{-1}(\{u\})$, $\mathbf{w}\in  \partial B_1^d$ and $\eps\in(0,1)$. Assume that  $\E[ \sigma_0(L_X(u) \cap l_{\mathbf w,\mathbf{0}} \cap [0,1]^d)^{1/\eps}]$ is finite, where $ l_{\mathbf w,\mathbf{0}}$ is defined in~\eqref{eqlines}.
\end{enumerate}
Note that if the random field $X$ is isotropic, the vector $\mathbf{w}\in  \partial B_1^d$ appearing in~\ref{A1} and~\ref{A2} can be chosen arbitrarily.   Assumption~\ref{A1} permits to  apply Lemma~\ref{lem:1D_small_sojourn}  to the unidirectional  process $Y$.  The aforementioned lemma   controls the probability that a one-dimensional field crosses a level $u$ more than once in a small time interval. Assumption~\ref{A2} is technical and allows for a H\"older control in the proof of Theorem \ref{prp:Pcrossing_Esurfacearea}.

\paragraph{A first result on crossings.}

The following key result provides a first order approximation of the probability of crossings for the isotropic random field  $X$ in a given arbitrary direction. This result plays a central role in deriving a formula for the expected surface area of the limiting approximation of an excursion set by elements of a point-referenceable $d$-honeycomb in the sense of Definition~\ref{def:lattice} (see Theorem~\ref{thm:honeycomb} below).

\begin{Theorem}\label{prp:Pcrossing_Esurfacearea}
Consider a scalar $q\in\R^+$ and a fixed $\mathbf{w}\in\partial B^d_1$. Let $X$ be an isotropic random field satisfying Assumption~\ref{A0}.
Then for any fixed  $u\in\R$,
\begin{equation}\label{eqn:first_item_thm2.1}
\lim_{q\to 0}\frac{1}{q}\P\big(X(\mathbf{0}) \leq u < X( q\mathbf{w})\big) = \frac{C_{d-1}^*(u)}{\beta_d},
\end{equation}
where \begin{align}\label{betad}
\beta_d = \frac{2\sqrt{\pi}\ \Gamma(\frac{d+1}{2})}{\Gamma(\frac{d}{2})}
\end{align}
is a dimensional constant and the limit in~\eqref{eqn:first_item_thm2.1} is approached from below. Moreover, if $X$ also satisfies assumptions~\ref{A1} and~\ref{A2} for $\mathbf{w}$ and some  $\eps\in(0,1)$, then  there exists a constant $K\in\R^+$ such that for all $q\in\R^+$,
\begin{equation}\label{eqn:second_item_thm2.1}
0\le \frac{C_{d-1}^*(u)}{\beta_d}q - \P\big(X(\mathbf{0}) \leq u < X( q\mathbf{w})\big)  \leq  K q^{2-\eps}.
\end{equation}
\end{Theorem}

The proof of Theorem~\ref{prp:Pcrossing_Esurfacearea} is based on the Crofton formula in~\eqref{eqn:crofton}.
\begin{proof}[Proof of Theorem~\ref{prp:Pcrossing_Esurfacearea}]
Firstly, we prove Equation~\eqref{eqn:first_item_thm2.1}.
We apply Equation~\eqref{eqn:crofton} to the $(d-1)$-manifold $L_X(u)\cap B_1^d$,
\begin{equation}\label{eqn:crofton_3}
\sigma_{d-1}\big(L_X(u)\cap B_1^d\big) = \frac{\sqrt{\pi}\ \Gamma(\frac{d+1}{2})}{\Gamma(\frac{d}{2})}
\int_{\partial B^d_1}\int_{\vect(\mathbf{s}^\perp)} \frac{\sigma_0(L_X(u)\cap B^d_1 \cap l_{\mathbf{s},\mathbf{v}})}{\sigma_{d-1}(\partial B^d_1)}\ \rmd \mathbf{v}\ \rmd\mathbf{s},
\end{equation}
with $l_{\mathbf{s},\mathbf{v}}$ as  in~\eqref{eqlines}.
The quantity in~\eqref{eqn:crofton_3} is a positive random variable in $L^1$, and thus by taking the expectation, we get
\begin{equation*} 
\E\big[\sigma_{d-1}\big(L_X(u)\cap B_1^d\big)\big] = \frac{\sqrt{\pi}\ \Gamma(\frac{d+1}{2})}{\Gamma(\frac{d}{2})}
\int_{\partial B^d_1}\int_{\vect(\mathbf{s}^\perp)} \frac{\E\big[\sigma_0(L_X(u)\cap B^d_1 \cap l_{\mathbf{s},\mathbf{v}})\big]}{\sigma_{d-1}(\partial B^d_1)}\ \rmd \mathbf{v}\ \rmd\mathbf{s}.
\end{equation*}
By the isotropy of $X$, we obtain
\begin{equation}\label{eqn:crofton_final}
\E\big[\sigma_{d-1}\big(L_X(u)\cap B_1^d\big)\big] = \frac{\sqrt{\pi}\ \Gamma(\frac{d+1}{2})}{\Gamma(\frac{d}{2})}
\int_{\vect(\mathbf{w}^\perp)}\E\big[\sigma_0(L_X(u)\cap B^d_1 \cap l_{\mathbf{w},\mathbf{v}})\big]\ \rmd\mathbf{v},
\end{equation}
for $l_{\mathbf{w},\mathbf{v}}$ as  in~\eqref{eqlines}.

For $\mathbf{v}\in \vect(\mathbf{w}^\perp)\cap B_1^d$, define $\bar{l_{\mathbf{w},\mathbf{v}}} := B^d_1\cap l_{\mathbf{w},\mathbf{v}}$. Then $\bar{l_{\mathbf{w},\mathbf{v}}}$ either contains a single point, or it is a line segment in $\R^d$ with orientation $\mathbf{w}$.
Define
$$n_q = \bigg\lfloor\frac{\sigma_1(\bar{l_{\mathbf{w},\mathbf{v}}})}{q}\bigg\rfloor,$$
where $\lfloor\cdot\rfloor$ denotes the floor function.
Define $\mathbf{x}_0 := \textbf{v} - \frac{\sigma_1(\bar{l_{\mathbf{w},\mathbf{v}}})}{2}\textbf{w}$  and $\mathbf{x}^0 := \textbf{v} + \frac{\sigma_1(\bar{l_{\mathbf{w},\mathbf{v}}})}{2}\textbf{w}$ to be the two endpoints of the line segment
$\bar{l_{\mathbf{w},\mathbf{v}}}$. Figure~\ref{fig:grassmanian} provides a graphical visualisation to aid this construction.

\begin{figure}[ht!]
    \centering
    \includegraphics[width=0.3\textwidth]{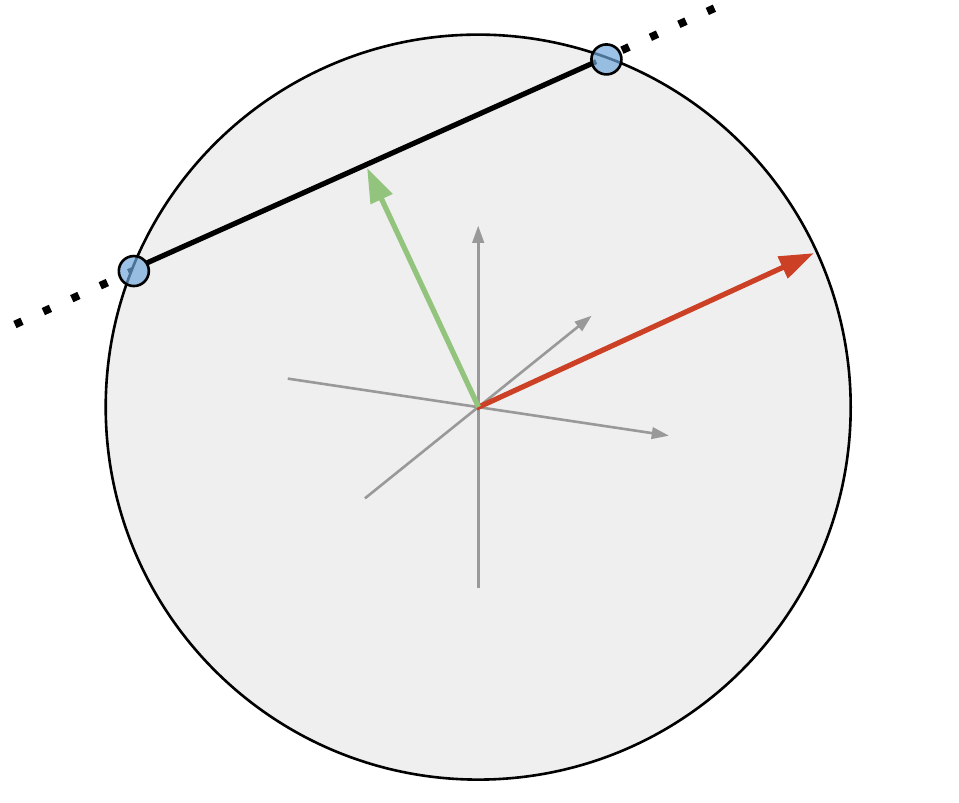}
    \put(-44,66){$\textbf{w}$}
    \put(-95,79){$\textbf{v}$}
    \put(-125,70){$\textbf{x}_0$}
    \put(-58,100){$\textbf{x}^0$}
    \caption{In dimension 3, the unit ball is represented; the red vector is a possible $\mathbf{w}\in \partial B_1^3$ and the green vector, a possible $\mathbf{v}\in\vect(\mathbf{w}^\perp)\cap B_1^3$. The two endpoints $\mathbf{x}_0$ and $\textbf{x}^0$, marked in blue, describe the segment $\bar{l_{\mathbf{w},\mathbf{v}}}$.} \label{fig:grassmanian}
\end{figure}
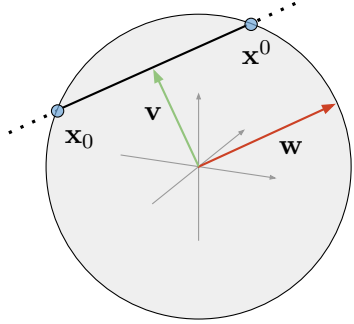

Now define $\mathbf{x}_j := \mathbf{x}_0+jq\mathbf{w}$ for $j\in \{0,...,n_q\}$, each of which being an element of $\bar{l_{\mathbf{w},\mathbf{v}}}$.
 With this construction, it follows that
\begin{equation}\label{eqn:as_to_chi}
\sum_{j=0}^{n_q-1}\big( \I{X(\mathbf{x}_j) \leq u < X(\mathbf{x}_{j+1})} + \I{X(\mathbf{x}_j) > u \geq X(\mathbf{x}_{j+1})}\big)
\stackrel{\mathrm{a.s.}}{\longrightarrow} \sigma_0(L_X(u)\cap \bar{l_{\mathbf{w},\mathbf{v}}}),
\end{equation}
as $q \rightarrow 0$.
The left-hand side of~\eqref{eqn:as_to_chi} is a count of the number of elements in $L_X(u)\cap \bar{l_{\mathbf{w},\mathbf{v}}}$, which approaches the exact value from below almost surely. This construction is well known in the literature as the \emph{discretization method}. The interested reader is referred for instance to \citet{KratzSurvey06} (Section 2.1) and references therein.
In particular, by using the stationarity of $X$ in Assumption~\ref{A0}, the monotonicity of the convergence in~\eqref{eqn:as_to_chi} implies
$$2n_q\P\big(X(\0) \leq u < X(\textbf{w}q)\big)\underset{q\rightarrow 0}{\longrightarrow} \E[\sigma_0(L_X(u)\cap \bar{l_{\mathbf{w},\mathbf{v}}})].$$
Moreover, this convergence is uniform for $\textbf{v}\in\vect(\textbf{w}^\perp)$, so
$$\P\big(X(\0) \leq u < X(\textbf{w}q)\big)\int_{\vect(\textbf{w}^\perp)\cap B_1^d}2n_q\rmd \textbf{v}\underset{q\rightarrow 0}{\longrightarrow} \int_{\vect(\textbf{w}^\perp)}\E[\sigma_0(L_X(u)\cap \bar{l_{\mathbf{w},\mathbf{v}}})]\rmd \textbf{v}.$$
By~\eqref{eqn:crofton_final}, this simplifies to
$$\P\big(X(\0) \leq u < X(\textbf{w}q)\big)\int_{\vect(\textbf{w}^\perp)\cap B_1^d} n_q\rmd \textbf{v} \underset{q\rightarrow 0}{\longrightarrow} \frac{1}{\beta_d}\E\big[\sigma_{d-1}\big(L_X(u)\cap B_1^d\big)\big].$$
Since $0 \leq \sigma_1(\bar{l_{\mathbf{w},\mathbf{v}}})/q - n_q < 1$
 and $\P\big(X(\0) \leq u < X(\textbf{w}q)\big)\underset{q\rightarrow 0}{\longrightarrow}0$, it holds that
$$\P\big(X(\0) \leq u < X(\textbf{w}q)\big)\int_{\vect(\textbf{w}^\perp)\cap B_1^d} \frac{\sigma_1(\bar{l_{\mathbf{w},\mathbf{v}}})}q\rmd \textbf{v} \underset{q\rightarrow 0}{\longrightarrow} \frac{1}{\beta_d}\E\big[\sigma_{d-1}\big(L_X(u)\cap B_1^d\big)\big].$$
By noticing that $\int_{\vect(\mathbf{w}^\perp)\cap B_1^d}{\sigma_{1}(\bar{l_{\mathbf{w},\mathbf{v}}})} \rmd \mathbf{v}=\sigma_{d}(B^{d}_1)$, one obtains
$$\frac{1}{q}\P\big(X(\0) \leq u < X(\textbf{w}q)\big) \underset{q\rightarrow 0}{\longrightarrow} \frac{1}{\beta_d\sigma_d(B_1^d)}\E\big[\sigma_{d-1}\big(L_X(u)\cap B_1^d\big)\big] = \frac{C_{d-1}^*(u)}{\beta_d}.$$
This proves the first order approximation in~\eqref{eqn:first_item_thm2.1}.

To show that
\begin{equation}\label{eqn:monotonicity_in_q}
\frac{C_{d-1}^*(u)}{\beta_d}q \geq \P\big(X(\mathbf{0}) \leq u < X( q\mathbf{w})\big)
\end{equation}
for all $q > 0$, suppose that there exists a $t\in\R^+$ such that
$$\frac{C_{d-1}^*(u)}{\beta_d}t < \P\big(X(\mathbf{0}) \leq u < X( t\mathbf{w})\big).$$
Then, for all $n\in\N^+$,
$\{X(\0) \leq  u < X(t\textbf{w})\} \subseteq \{X(\0) \leq  u < X(\tfrac{t}{n}\textbf{w})\}\cup \ldots \cup \{X(\tfrac{n-1}nt\textbf{w}) \leq  u < X(t\textbf{w})\},$
leading under~\ref{A0} to
$$\frac{C_{d-1}^*(u)}{\beta_d} < \frac{1}{t}  \P\big(X(\0) \leq  u < X(t\textbf{w})\big) \leq \frac{n}{t}  \P\big(X(\0) \leq  u < X(\tfrac{t}{n}\textbf{w})\big).$$ Keeping $t$ fixed and having $n\to\infty$
 contradicts the convergence in~\eqref{eqn:first_item_thm2.1}. Thus,~\eqref{eqn:monotonicity_in_q} holds for all $q\in\R^+$.

We proceed by showing~\eqref{eqn:second_item_thm2.1}.
For fixed $q>0$, there can only be a difference between the  expression in~\eqref{eqn:as_to_chi} and its limit---both taking values in $\N_0$---if at least one of the two following conditions hold,
\begin{enumerate}
\item[(i)]
There exists a $j\in\{0,\ldots,n_q-1\}$ such that the line segment with endpoints $\mathbf{x}_j$ and $\mathbf{x}_{j+1}$
contains two points $\mathbf{s}_1\neq \mathbf{s}_2$ such that $X(\mathbf{s}_1)=X(\mathbf{s}_2)=u$.
\item[(ii)] The line segment with endpoints $\textbf{x}_{n_q}$ and $\textbf{x}^0$ contains a point $\textbf{s}$ such that $X(\textbf{s})=u$.
\end{enumerate}
The probability of event in (i) is bounded above by $n_{q}K_2q^2$ for some $K_2\in\R^+$, under~\ref{A1} by applying Lemma~\ref{lem:1D_small_sojourn} with $p=2$. Furthermore, by using Markov's inequality and Theorem 3.3.1 in \cite{buetler1966}, the probability of event (ii) is bounded above by $K_1q$ for some $K_{1}\in \R^{+}$. Let $A$ denote the event $\sigma_0(L_X(u)\cap \bar{l_{\mathbf{w},\mathbf{v}}}) > \sum_{j=0}^{n_q-1}\big( \I{X(\mathbf{x}_j) \leq u < X(\mathbf{x}_{j+1})} + \I{X(\mathbf{x}_j) > u \geq X(\mathbf{x}_{j+1})}\big)$, then  applying H\"older's inequality together with~\ref{A2} implies
\begin{align*}
\E\big[&\sigma_0(L_X(u)\cap \bar{l_{\mathbf{w},\mathbf{v}}})\big] - 2n_q\P\big(X(\mathbf{0}) \leq u < X(q\mathbf{w})\big)\\
&= \E\Big[\Big|\sigma_0(L_X(u)\cap \bar{l_{\mathbf{w},\mathbf{v}}}) - \sum_{j=0}^{n_q-1}\big( \I{X(\mathbf{x}_j) \leq u < X(\mathbf{x}_{j+1})} + \I{X(\mathbf{x}_j) > u \geq X(\mathbf{x}_{j+1})}\big) \Big|\Big]\\
&\le  \E\big[\sigma_0(L_X(u)\cap \bar{l_{\mathbf{w},\mathbf{v}}})\I{A} \big]\le \left(\E\big[\big(\sigma_0(L_X(u)\cap \bar{l_{\mathbf{w},\mathbf{v}}})\big)^{\frac1\eps} \big]\right)^{\eps} \PP(A)^{1-\eps}\\
&\leq \left(\E\big[\big(\sigma_0(L_X(u)\cap \bar{l_{\mathbf{w},\mathbf{v}}})\big)^{\frac1\eps} \big]\right)^{\eps}\big(n_q K_2q^2 + K_1q\big)^{1-\eps}\\
&\leq \left(\E\big[\big(\sigma_0(L_X(u)\cap \bar{l_{\mathbf{w},\mathbf{v}}})\big)^{\frac1\eps} \big]\right)^{\eps}\big(\sigma_{1}(\bar{l_{\mathbf{w},\mathbf{v}}})K_2 + K_1\big)^{1-\eps}q^{1-\eps}\\
&\leq\left(\E\big[\big(\sigma_0(L_X(u)\cap \bar{l_{\mathbf{w},\mathbf{v}}})\big)^{\frac1\eps} \big]\right)^{\eps}\big(2 K_2 + K_1\big)^{1-\eps}q^{1-\eps}.
\end{align*}
Furthermore, note that
\begin{align*}
\Big|\E\big[&\sigma_0(L_X(u)\cap \bar{l_{\mathbf{w},\mathbf{v}}})\big] - 2\frac{\sigma_{1}(\bar{l_{\mathbf{w},\mathbf{v}}})}{q}\P\big(X(\mathbf{0})\leq u < X(q\mathbf{w})\big)\Big|\\
&= \Big|\E\big[\sigma_0(L_X(u)\cap \bar{l_{\mathbf{w},\mathbf{v}}})\big] - 2n_q\P\big(X(\mathbf{0})\leq u < X(q\mathbf{w})\big)+ 2\Big(n_q-\frac{\sigma_{1}(\bar{l_{\mathbf{w},\mathbf{v}}})}{q}\Big)\P\big(X(\mathbf{0})\leq u < X(q\mathbf{w})\big)\Big|\\
&\leq \left(\E\big[\big(\sigma_0(L_X(u)\cap \bar{l_{\mathbf{w},\mathbf{v}}})\big)^{\frac1\eps} \big]\right)^{\eps}\big(2K_2 + K_1\big)^{1-\eps}q^{1-\eps} + 2\P\big(X(\mathbf{0})\leq u < X(q\mathbf{w})\big)\\
&\leq \left(\E\big[\big(\sigma_0(L_X(u)\cap \bar{l_{\mathbf{w},\mathbf{v}}})\big)^{\frac1\eps} \big]\right)^{\eps}\big(2K_2 + K_1\big)^{1-\eps}q^{1-\eps} + 2K_1q.
\end{align*}
Summarizing, we have shown that
$$\Big|q\E\big[\sigma_0(L_X(u)\cap \bar{l_{\mathbf{w},\mathbf{v}}})\big] - 2{\sigma_{1}(\bar{l_{\mathbf{w},\mathbf{v}}})}\P\big(X(\mathbf{0})\leq u < X(q\mathbf{w})\big)\Big| = O(q^{2-\eps}),$$
for all $\mathbf{v}\in \vect(\mathbf{w}^\perp)\cap B_1^d$.
By integrating over $\mathbf{v}$ and taking the absolute value, one preserves
\begin{equation}\label{eqn:Oq2}
\Big|\int_{\vect(\mathbf{w}^\perp)\cap B_1^d}\Big( q\E\big[\sigma_0(L_X(u)\cap \bar{l_{\mathbf{w},\mathbf{v}}})\big] - 2{\sigma_{1}(\bar{l_{\mathbf{w},\mathbf{v}}})}\P\big(X(\mathbf{0})\leq u < X(q\mathbf{w})\big)\Big)\ \rmd \mathbf{v}\Big| = O(q^{2-\eps}).
\end{equation}
Evaluating the integral, and again applying $\int_{\vect(\mathbf{w}^\perp)\cap B_1^d}{\sigma_{1}(\bar{l_{\mathbf{w},\mathbf{v}}})} \rmd \mathbf{v}=\sigma_{d}(B^{d}_1) $, one obtains
\begin{align*}
\int_{\vect(\mathbf{w}^\perp)\cap B_1^d}&\Big( q\E\big[\sigma_0(L_X(u)\cap \bar{l_{\mathbf{w},\mathbf{v}}})\big] - 2{\sigma_{1}(\bar{l_{\mathbf{w},\mathbf{v}}})}\P\big(X(\mathbf{0})\leq u < X(q\mathbf{w})\big)\Big)\ \rmd \mathbf{v}\\
&= q\int_{\vect(\mathbf{w}^\perp)} \E\big[\sigma_0(L_X(u)\cap \bar{l_{\mathbf{w},\mathbf{v}}})\big]\ \rmd \mathbf{v} -2\sigma_d(B_1^d)\P\big(X(\mathbf{0})\leq u < X(q\mathbf{w})\big)\\
&= \frac{2q}{\beta_d}\E\big[\sigma_{d-1}\big(L_X(u)\cap B_1^d\big)\big] - 2\sigma_d(B_1^d)\P\big(X(\mathbf{0})\leq u < X(q\mathbf{w})\big)\\
&= 2\sigma_d(B_1^d)\Big( \frac{C_{d-1}^*(u)}{\beta_d}q- \P\big(X(\mathbf{0}) \leq u < X( q\mathbf{w})\big)\Big),
\end{align*}
where the second equality follows from~\eqref{eqn:crofton_final}. Thus, by~\eqref{eqn:Oq2}, we obtain the desired result
$$\Big|\frac{C_{d-1}^*(u)}{\beta_d}q - \P\big(X(\mathbf{0}) \leq u < X( q\mathbf{w})\big)\Big| = O(q^{2-\eps}).$$
\end{proof}

\begin{Remark}
Theorem~\ref{prp:Pcrossing_Esurfacearea} can also be  very useful for sample simulations. To have a rapid evaluation of the surface area of excursion sets of a given random field satisfying above assumptions, it is not necessary to generate the whole isotropic random field on $\R^{d}$ but only \emph{i.i.d}. r.v. with the same bivariate distribution as $\big(X({\0}), X({q\,\be_{1}}) \big)$, for $q$ small enough. Then, the numerical first order approximation of the surface area would be $$ \beta_{d}\frac{\hat\PP\big(X(\mathbf{0}) \leq u < X( q\be_{1})\big)}{q}.$$
This induces an error of the order $q^{1-\eps}$ according to Equation~\eqref{eqn:second_item_thm2.1}. This approach can be compared for instance with the tedious computations encountered when using  tube formulas (see, \emph{e.g.}, \cite{AT07}) to obtain an exact value of the expected surface area.
\end{Remark}

\paragraph{A related result.}

Notice that Theorem~\ref{prp:Pcrossing_Esurfacearea} fully identifies the limit of the level crossing and relies on classical assumptions. Another result  on level crossings can be obtained by adapting a result in \cite{Leadbetter} to dimension $d$ (see Proposition~\ref{LeadbetterTheorem} below). For clarity, the proof of this adaptation in general dimension $d$ is given in Section~\ref{prf:secSA}.

\begin{Proposition}[A $d-$dimensional formulation of \cite{Leadbetter}, Theorem 7.2.4]\label{LeadbetterTheorem} Let $X:\R^{d}\to\R$ be a continuous stationary random field  such that $X(\0)$ and $X_{q}:=\frac1q(X(q \be_{1})-X(\0))$ have a joint density denoted $g_{q}(u,z)$ that is continuous in $u$ for all $z$ and all sufficiently small $q>0$, and that there exists a function $p(u,z)$ such that $g_{q}(u,z)\rightarrow p(u,z)$ uniformly in $u$ for fixed $z$ as $q\to0$.
Assume furthermore that there is a function $h(z)$ such that $\int_{0}^{\infty} zh(z)\rmd z<\infty$ and $g_{q}(u,z)\le h(z)$ {uniformly} for all $u, q$. Then, it holds that
$$\lim_{q\to0} \frac{\PP\left(X(\0)\leq  u < X(q \be_1)\right)}{q}=\int_{0}^{\infty}z\,p(u,z)\rmd z.$$
\end{Proposition}

Remark that it is difficult to apply Proposition~\ref{LeadbetterTheorem} in practice, since computing the asymptotic joint density $p$ is not easy outside of specific cases such as the Gaussian one. Moreover, it requires that the convergence of the density of $X_{q}$ as $q\to0$ is established. Even if it is clear that $X_{q}$ converges almost surely to $\partial_1 X(\0)$,  where $\partial_{1}X$ denotes the partial derivative of $X$ in the direction $\be_1$, having the convergence of the corresponding  densities is  much more delicate to establish (see, \emph{e.g.}, \cite{boos1985converse} or \cite{sweeting1986converse}).  If $d=1$, \cite{Leadbetter} explains that ``\emph{In many cases, the limit $p(u,z)$ is simply the joint density of $\left(X(\0),\partial_{1}X(\0)\right)$.}'' This holds for example, for Gaussian processes. Translating this in dimension $d\ge 1$, if we write $p_{X(\0)}(u)$  for the density of $X(\0)$, we see that
\begin{align}\label{eq:ClaimP}
\lim_{q\to0} \frac{\PP\left(X(\0)\leq  u< X(q  \be_1)\right)}{q}=p_{X(\0)}(u)\E\left[\partial_1 X(\0) \mathbbm{1}_{\{\partial_1 X(\0)>0\}}|X(\0)=u\right].
\end{align}

\paragraph{Bias  control for the estimated surface area measure.}

The following theorem provides the explicit formula for the bias of the surface area in any dimension $d$ and for arbitrary point-referenced $d$-honeycombs in Definition~\ref{def:lattice} whose polytopes do not grow too large near $\infty$. In this way it generalizes and unifies existing results (see, \textit{e.g.}  Proposition 5 in \citet{bierme:hal-02793752} for the square and hexagonal tiling  in dimension $2$ and \citet{miller1999} for the triangular tiling in dimension 2 and hypercubic lattice in dimension 3). Furthermore, Theorem~\ref{thm:honeycomb} applies to Poisson-Voronoi tessellations, as shown in Corollary~\ref{corollaryBiais}.

\begin{Theorem}\label{thm:honeycomb} Let $T\subset \R^d$ be a compact domain  with non empty interior containing the origin. Let $X$ be an isotropic random field on $\R^d$ that satisfies Assumption~\ref{A0}.

\begin{itemize}
\item[i)] Let $\mathcal{H}\in\mathfrak{T}^d$  (see Definition~\ref{def:lattice}) and let $\dot{\mathcal{H}}$ be a corresponding point-referenced $d$-honeycomb. Define $D^{(\mathcal{H})} := \sup\{\mathrm{diam}(P\cap T) : P\in\mathcal{H}\}$ and $\delta\mathcal{H}:= \{\delta P : P \in \mathcal{H}\}$ for $\delta\in\R^+$.  Let $\widehat C^{(\dot{\mathcal{H}},T)}_{d-1}(u)$ be  as in Definition~\ref{def:estGenTil}.
Suppose  that $\lim_{\delta\to 0}D^{(\delta\mathcal{H})} = 0$.
Then, it holds
\begin{equation}\label{eqn:honeycomb_deterministic}\E\big[\widehat C^{(\delta \dot{\mathcal{H}},T)}_{d-1}(u)\big] \underset{\delta\to 0}{\longrightarrow} \frac{2d}{\beta_d}C^*_{d-1}(u),
\end{equation}
where $\delta\dot{\mathcal{H}} := \{(\delta P, \delta P^\bullet) : (P,P^\bullet)\in\dot{\mathcal{H}}\}$ is $\dot{\mathcal{H}}$ linearly rescaled by $\delta$, $\beta_d$ is as in~\eqref{betad}, and $C^*_{d-1}(u)$ is as in~\eqref{defCast}.

\item[ii)] Furthermore, if $\xi$ is a point process on $\R^d$ that is independent of $X$,  let $\mathcal{V}_\xi$ be the Voronoi diagram in $\mathfrak{T}^d$ generated from the points in $\xi$. Reference the polytopes in $\mathcal{V}_\xi$ by their corresponding points in $\xi$ and denote the resulting point-referenced $d$-honeycomb by $\dot{\mathcal{V}}_\xi$. If $D^{(\delta\mathcal{V}_\xi)} \stackrel{\P}\to 0$ as $\delta\to 0$, then
\begin{equation*}\label{eqn:honeycomb_PP}
\E\big[\widehat C^{(\delta \dot{\mathcal{V}}_\xi,T)}_{d-1}(u)\big] \underset{\delta\to 0}{\longrightarrow} \frac{2d}{\beta_d}C^*_{d-1}(u).
\end{equation*}
\end{itemize}
\end{Theorem}

\begin{proof}[Proof of Theorem~\ref{thm:honeycomb}]
First, we establish Equation~\eqref{eqn:honeycomb_deterministic} where the point-referenced $d$-honeycomb $\dot{\mathcal{H}}$ is deterministic. The estimator $\widehat C^{(\delta\dot{\mathcal{H}},T)}_{d-1}(u)$ in~\eqref{GeneralLatticeEstimaor} is written
$$\widehat C^{(\delta\dot{\mathcal{H}},T)}_{d-1}(u) = \frac{1}{\sigma_d(T)}
\underset{(P_1\neq P_2),  (\delta P_1,\delta P_2\subseteq T)}{\sum_{P_1,P_2\in\mathcal{H}}}
\sigma_{d-1}(\delta P_1\cap \delta P_2)\I{X(\delta P_1^\bullet) \leq  u < X(\delta P_2^\bullet)}.$$
By the linearity of the expectation, for $\delta > 0$,
\begin{equation}\label{eqn:linearity_E_lattice}
\E[\widehat C^{(\delta\dot{\mathcal{H}},T)}_{d-1}(u)] =
\frac{1}{\sigma_d(T)}
\underset{(P_1\neq P_2),  (\delta P_1,\delta P_2\subseteq T)}{\sum_{P_1,P_2\in\mathcal{H}}}
\sigma_{d-1}(\delta P_1\cap \delta P_2)\P\big(X(\delta P_1^\bullet) \leq  u < X(\delta P_2^\bullet)\big).
\end{equation}
Let $P_1,P_2\in \mathcal{H}$ be such that $P_1 \neq P_2$ and $\sigma_{d-1}(P_1\cap P_2) > 0$, \textit{i.e.}, such that $P_1$ and $P_2$ are adjacent in $\mathcal{H}$. Then as $\delta \to 0$, clearly $||\delta P_2^\bullet - \delta P_1^\bullet||_2\to 0$, and so
by Equation~\eqref{eqn:first_item_thm2.1} in Theorem~\ref{prp:Pcrossing_Esurfacearea},
\begin{equation}\label{eqn:uniform_converge_lattice}
\frac{\P\big(X(\delta P_1^\bullet) \leq  u < X(\delta P_2^\bullet)\big)}{||\delta P_2^\bullet - \delta P_1^\bullet||_2} \underset{\delta\to 0}\longrightarrow \frac{C^*_{d-1}(u)}{\beta_d}.
\end{equation}
Moreover, if one restricts to neighbouring polytopes $P_1$ and $P_2$ in $\mathcal{H}$ such that $\delta P_1,\delta P_2 \subseteq T$, then $D^{(\delta\mathcal{H})}\to 0$ implies that the convergence in~\eqref{eqn:uniform_converge_lattice} is uniform as $T$ remains fixed, \textit{i.e.}
$$\sup\bigg\{\Big|\frac{\P\big(X(\delta P_1^\bullet) \leq  u < X(\delta P_2^\bullet)\big)}{||\delta P_2^\bullet - \delta P_1^\bullet||_2} - \frac{C^*_{d-1}(u)}{\beta_d}\Big| : P_1\mathrm{\ neighbours\ } P_2,\mathrm{\ and\ } \delta P_1,\delta P_2\subseteq T\bigg\}\underset{\delta\to 0}\longrightarrow 0.
$$
Therefore, Equation~\eqref{eqn:linearity_E_lattice} can be written
\begin{equation*}
\E[\widehat C^{(\delta\dot{\mathcal{H}},T)}_{d-1}(u)] - \frac{C^*_{d-1}(u)}{\beta_d\sigma_d(T)}
\underset{(P_1\neq P_2),  (\delta P_1,\delta P_2\subseteq T)}{\sum_{P_1,P_2\in\mathcal{H}}}
\sigma_{d-1}(\delta P_1\cap \delta P_2)||\delta P_2^\bullet - \delta P_1^\bullet||_2 \underset{\delta\to 0}\longrightarrow 0.
\end{equation*}
With the change of variables $Q_1 = \delta P_1$ and $Q_2 = \delta P_2$, we have
\begin{equation}\label{eqn:conditional_exp_converges_1}
\E[\widehat C^{(\delta\dot{\mathcal{H}},T)}_{d-1}(u)] -
\frac{C^*_{d-1}(u)}{\beta_d\sigma_d(T)}
\underset{Q_1\neq Q_2}{\sum_{Q_1,Q_1\in(\delta\mathcal{H})^T}}
\sigma_{d-1}( Q_1\cap Q_2)||Q_2^\bullet - Q_1^\bullet||_2 \underset{\delta\to 0}\longrightarrow 0.
\end{equation}
For adjacent $Q_1,Q_2\in(\delta\mathcal{H})^T$, denote by $\Lambda_{Q_1,Q_2}$ the $d$-dimensional hyperpyramid with base $Q_1\cap Q_2$ and vertex $Q_1^\bullet$ (see Figure~\ref{fig:pyramids}).

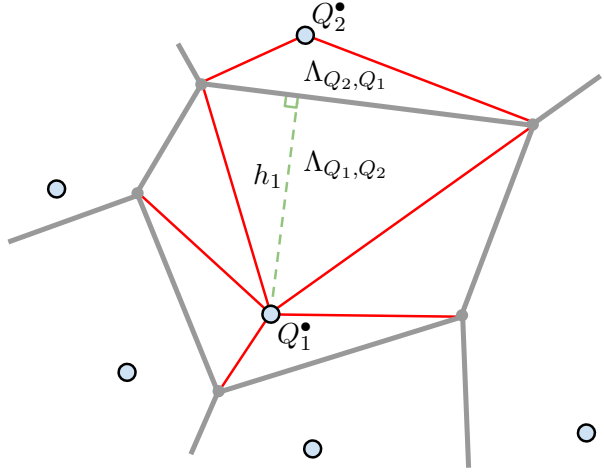
\begin{figure}[ht!]
    \centering
    \includegraphics[width=0.5\textwidth]{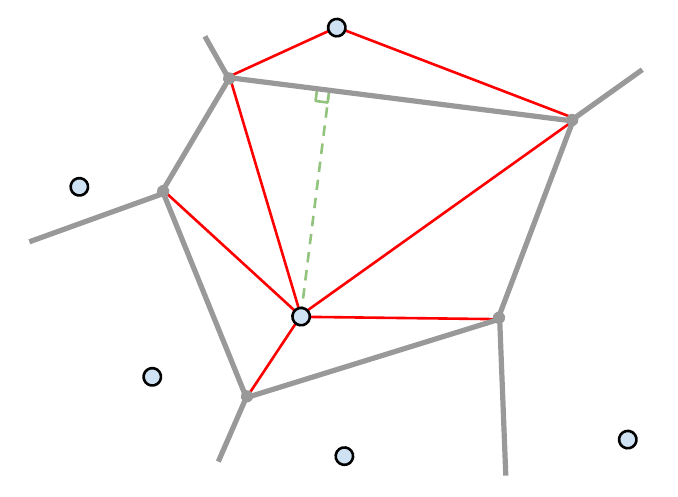}
    \put(-125,120){$\Lambda_{Q_1,Q_2}$}
    \put(-125,154){$\Lambda_{Q_2,Q_1}$}
    \put(-144,115){$h_1$}
    \put(-122,176){$Q_2^{\bullet}$}
    \put(-135,55){$Q_1^{\bullet}$}

    \caption{A pentagonal cell $Q_1$ in a point-referenced 2-honeycomb is partitioned into five pyramids (triangles in the case $d=2$) each with summit $Q_1^\bullet$. A sixth pyramid $\Lambda_{Q_2,Q_1}$ with summit $Q_2^\bullet$ is drawn in the cell $Q_2$, which shares its base $Q_1\cap Q_2$ with the pyramid $\Lambda_{Q_1,Q_2}$ contained in $Q_1$. The height of the pyramid $\Lambda_{Q_1,Q_2}$ is marked $h_1$.
 \label{fig:pyramids}}
\end{figure}

 The $d$ dimensional analogue of the area of a triangle implies that $\sigma_d(\Lambda_{Q_1,Q_2})=\frac{1}{d}\sigma_{d-1}(Q_1\cap Q_2) h_1$, where $h_1$ is the distance from $Q_1^\bullet$ to the hyperplane containing $Q_1\cap Q_2$. Now, $\sigma_d(\Lambda_{Q_1,Q_2} \cup \Lambda_{Q_2,Q_1})=\frac{1}{d}\sigma_{d-1}( Q_1\cap Q_2)||Q_2^\bullet -  Q_1^\bullet||_2$, and so
\begin{align*}
\underset{Q_1\neq Q_2}{\sum_{Q_1,Q_1\in(\delta\mathcal{H})^T}}
\sigma_{d-1}( Q_1\cap Q_2)||Q_2^\bullet - Q_1^\bullet||_2 &= \hspace{-0.41cm}
\underset{Q_1\mbox{\ adjacent\ to\ }Q_2}{\sum_{Q_1,Q_2\in(\delta\mathcal{H})^T}}
\hspace{-0.41cm} d\, \sigma_d(\Lambda_{Q_1,Q_2} \cup \Lambda_{Q_2,Q_1})=
2d \hspace{-0.41cm}
\underset{Q_1\mbox{\ adjacent\ to\ }Q_2}{\sum_{Q_1,Q_2\in(\delta\mathcal{H})^T}}
\hspace{-0.41cm} \sigma_d(\Lambda_{Q_1,Q_2})\\
&=
2d\,
\sigma_d\Bigg(\underset{Q_1\mbox{\ adjacent\ to\ }Q_2}{\bigcup_{Q_1,Q_2\in(\delta\mathcal{H})^T}}
\Lambda_{Q_1,Q_2}\Bigg).
\end{align*}
Since every point in $T$ at distance of at least $2D^{(\delta\mathcal{H})}$ from the boundary of $T$ is contained in a $\Lambda_{Q_1,Q_2}$ for some $Q_1,Q_2\in(\delta\mathcal{H})^T$, and $\Lambda_{Q_1,Q_2}$ is contained in $Q_1\subseteq T$,
$$\lim_{\delta\to 0}\underset{Q_1\mbox{\ adjacent\ to\ }Q_2}{\bigcup_{Q_1,Q_2\in(\delta\mathcal{H})^T}}
\Lambda_{Q_1,Q_2} = T.$$
Thus, by the continuity of the measure $\sigma_d$,
\begin{equation}\label{eqn:sigma_d_continuous}
\underset{Q_1\neq Q_2}{\sum_{Q_1,Q_1\in(\delta\mathcal{H})^T}}
\sigma_{d-1}( Q_1\cap Q_2)||Q_2^\bullet - Q_1^\bullet||_2 \underset{\delta\to 0}\longrightarrow 2d\,\sigma_d(T),
\end{equation}
and by~\eqref{eqn:conditional_exp_converges_1},
$$\E[\widehat C^{(\delta\dot{\mathcal{H}},T)}_{d-1}(u)] \underset{\delta\to 0}\longrightarrow \frac{2d}{\beta_d}C^*_{d-1}(u).$$
This proves~\eqref{eqn:honeycomb_deterministic}. Now if the tessellation is produced from a point process $\xi$ and if $D^{(\delta\mathcal{V}_\xi)}\stackrel{\P}\to 0$ as $\delta\to 0$, then~\eqref{eqn:honeycomb_deterministic} implies that
\begin{equation}\label{eqn:as_convergence_honeycomb}
\E[\widehat C^{(\delta\dot{\mathcal{V}}_\xi,T)}_{d-1}(u)\ |\ \xi] \stackrel{\P}\longrightarrow \frac{2d}{\beta_d}C^*_{d-1}(u),\qquad\delta\to 0.
\end{equation}
What remains to be shown is that the convergence in~\eqref{eqn:as_convergence_honeycomb} holds in $L^1$.  First, remark that the sequence in~\eqref{eqn:uniform_converge_lattice} converges to its limit from below by Theorem~\ref{prp:Pcrossing_Esurfacearea}.
Remark also that the sequence in~\eqref{eqn:sigma_d_continuous} converges to its limit from below. These two remarks, together with the convergence in probability in~\eqref{eqn:as_convergence_honeycomb}, imply that
$$\sup_{\delta > 0} \E[\widehat C^{(\delta\dot{\mathcal{V}}_\xi,T)}_{d-1}(u)\ |\ \xi] = \frac{2d}{\beta_d}C^*_{d-1}(u),\as$$
Therefore, the convergence in~\eqref{eqn:as_convergence_honeycomb} holds in $L^1$ and
$\E\big[\E[\widehat C^{(\delta\dot{\mathcal{V}}_\xi,T)}_{d-1}(u)\ |\ \xi]\big] \underset{\delta\to 0}\longrightarrow \frac{2d}{\beta_d}C^*_{d-1}(u)$
as desired.
\end{proof}

\begin{Remark}
In the second part of Theorem~\ref{thm:honeycomb}, the assumption that the point process $\xi$---that generates the point cloud of sample locations---is independent of $X$ is crucial. Otherwise, for any realization $X$, it would be possible to build a dependent point-referenced $d$-honeycomb adapted to the  realized $L_{X}^{T}(u)$ such that $\E\big[\widehat C^{(\delta \dot{\mathcal{V}}_\xi,T)}_{d-1}(u)\big] \underset{\delta\to 0}{\longrightarrow} C^*_{d-1}(u)$ and no corrective constant is required (see \cite{cotsakis2022perimeter}, for $d=2$). On the contrary, with a deterministic point-referenced $d$-honeycomb, the dimensional constant  $2d/\beta_{d}$ is unavoidable.
\end{Remark}

In the following corollary, we apply Theorem~\ref{thm:honeycomb} to Poisson-Voronoi tessellations.
The interested reader is also referred to Theorem~1.1 in \cite{thale2016asymptotic}.

\begin{Corollary}\label{corollaryBiais}
Let $T\subset \R^d$ be a compact domain with non empty interior containing the origin. Let $X$ be an isotropic random field on $\R^d$ satisfying Assumption~\ref{A0}, and let $\xi$ be a homogeneous Poisson point process on $\R^d$ with unit rate, independent of $X$. Let $\mathcal{V}_\xi$ be the Poisson-Voronoi tessellation constructed from $\xi$, thus making $\dot{\mathcal{V}}_\xi$ a random point-referenced $d$-honeycomb. With $\delta > 0$ and $\widehat C^{(\delta\dot{\mathcal{V}}_\xi,T)}_{d-1}(u)$ defined as in Definition~\ref{def:estGenTil}, it holds that
\begin{equation*}
\E\big[\widehat C^{(\delta \dot{\mathcal{V}}_\xi,T)}_{d-1}(u)\big] \stackrel{\delta\to 0}{\longrightarrow} \frac{2d}{\beta_d}C^*_{d-1}(u).
\end{equation*}
\end{Corollary}

\begin{proof}
By Theorem~\ref{thm:honeycomb}, it suffices to show that $D^{(\delta\mathcal{V}_\xi)}\stackrel{\P}\to 0$ as $\delta\to 0$, where $D^{(\delta\mathcal{V}_\xi)} = \sup\{\mathrm{diam}(P\cap T):P\in\delta\mathcal{V}_\xi\}$. Let $1\ge \epsilon > 0$, denote  by $A_\epsilon$  the event $\{D^{(\delta\mathcal{V}_\xi)} > \epsilon\}$. Note that
\begin{align*}A_\epsilon = \{\exists P \in \delta\mathcal{V}_\xi : \mathrm{diam}(P\cap T) > \epsilon\}&\subset\left\{\exists r\in \frac1\delta (T+1), (r+B^{d}_{\frac{\eps}{2\delta}})\cap \xi=\emptyset\right\}\\ &\subset\bigcup_{r\in \frac{\eps}{4}\Z\cap \frac{T+1}{\delta}}\left\{ (r+B^{d}_{\frac{\eps}{4\delta}})\cap \xi=\emptyset\right\},\end{align*}
 where $T+1=\{r\in \R^{d},\ d(r,T)\le 1\}$. It follows from the stationarity of $\xi$ and the compactness of $T$ that $$\PP(A_{\xi})\le \frac{\sigma_{d}(T+1)4^{d}}{(\eps\delta)^{d}}\PP\left(B^{d}_{\frac{\eps}{4\delta}}\cap \xi=\emptyset\right)= \frac{\sigma_{d}(T+1)4^{d}}{(\eps\delta)^{d}}e^{-\frac{\pi^{d/2}}{\Gamma(\frac d2+1)}\left(\frac{\eps}{4\delta}\right)^{d}}\to0\mbox{ as }\delta\to0.$$
\end{proof}

\begin{Remark}[Convergence of the bias factor]\label{RemarkBiasFactor}
 Theorem~\ref{thm:honeycomb} \emph{i}) offers insight about the limiting value of $\E\big[\widehat C^{(\delta \dot{\mathcal{H}},T)}_{d-1}(u)\big]$ as $\delta\to 0$. However, this result does not imply the convergence of the bias factor to $2d/\beta_d$.   It can be shown by slightly modifying the proofs of Theorems~\ref{prp:Pcrossing_Esurfacearea} and~\ref{thm:honeycomb} that
\begin{equation}\label{eqn:bias_converge}
E\bigg[\frac{\widehat C^{(\delta \dot{\mathcal{H}},T)}_{d-1}(u)}{C^T_{d-1}(u)}\ \Big|\ C^T_{d-1}(u) > 0\bigg]\underset{\delta\rightarrow 0}{\longrightarrow} \frac{2d}{\beta_d},
\end{equation}
with  $X$ an isotropic random field on $\R^d$ satisfying Assumption~\ref{A0}  and  $\dot{\mathcal{H}}$  a point-referenced $d$-honeycomb as in Theorem~\ref{thm:honeycomb} \emph{i}). A  sketch  of the proof of the convergence in~\eqref{eqn:bias_converge} is provided in Appendix~\ref{sec:convergence_of_bias}.
\end{Remark}

\section{Joint central limit theorem for the hypercubic  lattice}\label{sec:TCL}

In this section we  prove a joint central limit theorem for the estimated volume and surface area of excursion sets of $d-$dimensional isotropic  smooth random fields observed over a  hypercubic lattice. Firstly, we rewrite estimators in~\eqref{GeneralLatticeEstimaor}-\eqref{eqCdHoneycomb} in this specific setting.

\subsection{Estimators for the hypercubic lattice}\label{hypercubic:section}

For $\textbf{i}\in\Z^d$ and $\delta \in\R^+$, let $V_{\ii}(\delta) := \delta \textbf{i} + [0,\delta]^d$. Define $\delta\mathcal{G} := \{V_{\ii}(\delta) : \ii\in\Z^d\} \in \mathfrak{T}^d$ (see Definition~\ref{def:lattice}) and the point-referenced $d$-honeycomb
$$\delta\dot{\mathcal{G}} := \{(V_{\ii}(\delta),\delta \textbf{i}): \textbf{i} \in \Z^d\}.$$
Define
\begin{align}\label{G}
\mathbb{G}(\delta,T) := \{\delta\ii : \ii\in\Z^d,V_{\ii}(\delta) \subseteq T\}\subset\delta\Z^d.
\end{align} 
 Then, referring to Definition~\ref{def:estGenTil}, we write
\begin{align}\label{eq:C2pixel}
\widehat C_{d}^{(\delta\dot{\mathcal{G}},T)}(u) &= \frac{\delta^d}{\sigma_d(T)}\sum_{t\in\mathbb{G}(\delta,T)} \I{X(t)\geq u},
 \end{align}
and
\begin{align}\label{eq:C1tildeNEW}
 \widehat C_{d-1}^{(\delta\dot{\mathcal{G}},T)}(u) = \frac{\delta^{d-1}}{\sigma_d(T)}
\sum_{j=1}^d\underset{t+\delta\be_j\in\mathbb{G}(\delta,T)}{\sum_{t\in\mathbb{G}(\delta,T)}} \big|\I{X(t) \geq u} - \I{X(t+\delta\be_j) \geq u}\big|. \end{align}

In the following, without loss of generality, we assume that  $T$ is centered at the origin.
Furthermore, we suppose that $0 < \delta \leq 1$ is chosen such that $T = [-\delta N, \delta N]^d$ for some $N\in\N^+$, which implies that $\bigcup_{P\in(\delta\mathcal{G})^T}P = T$, and that  the Card$\left(\mathbb G(\delta, T)\right)=(2N)^{d}$  and $\sigma_d(T)= (2 N \delta)^d$.   Remark that this constriction implies that the deterministic bias ratio in Equation~\eqref{eqbiaisarea}  is exactly unity.  \smallskip

In the sequel we simplify the  notation of Equations~\eqref{eq:C2pixel}-\eqref{eq:C1tildeNEW}   by writing $\widehat{C}_d^{(\delta, T)}(u)$ and $\widehat{C}_{d-1}^{(\delta, T)}(u)$, respectively.  Furthermore,  we write $T=T_N$ when the dependence in $N$ needs to be explicitly specified.

\subsection{The dominant role of the $L_1$-norm for the hypercubic lattice}

Note that, $C^T_{d-1}(u)$ in~\eqref{eq:UestCT1bis} can be rewritten,  by the Crofton formula in~\eqref{eqn:crofton}, as
\begin{equation}\label{eqn:P2_crofton}
C^T_{d-1}(u)  =  \frac{1}{\sigma_d(T)}\int_{L_{X}^{T}(u)}\sigma_{d-1}(\rmd s)=   \frac{\sqrt{\pi}\ \Gamma(\frac{d+1}{2})}{\sigma_d(T)\Gamma(\frac{d}{2})}
\int_{\R^{d-1}}\int_{\partial B^d_1} \frac{\sigma_0(L^{T}_X(u) \cap l_{\mathbf{s},\mathbf{v}_{\mathbf{s}}(\mathbf{u})})}{\sigma_{d-1}(\partial B^d_1)}\ \rmd \mathbf{s}\ \rmd\mathbf{u}, \, \,\, \, a.s.
\end{equation}
 A first intuition is that when $\delta$ gets small, the estimated surface area~\eqref{eq:C1tildeNEW} gets close  to~\eqref{eqn:P2_crofton}. However, this is incorrect and  we show in the following result that when $\delta$ gets small our estimate is close in the $L^{1}$ sense to the following random variable
\begin{align} \label{def:Pp}
\widetilde{C}^T_{d-1}(u)&:=\frac{1}{\sigma_d(T)} \int_{L^{T}_X(u)} \frac{\|\nabla X(s)\|_1}{\|\nabla X(s)\|_2}\ \sigma_{d-1}(\rmd s) =  \frac{1}{\sigma_d(T)} \int_{T}  \delta_0\big(X(t)-u\big)||\nabla X(t)||_1\, \rmd t,
\end{align}
where $\delta_0(\cdot)$ denotes the Dirac delta distribution.
The last equality is obtained by the well-known Coarea formula (see, \emph{e.g.}, Equation~(7.4.14) in \cite{AT07}). 
The difference between~\eqref{eqn:P2_crofton} and~\eqref{def:Pp}, that induces the asymptotic bias, is the ratio ${\|\nabla X(s)\|_1}/{\|\nabla X(s)\|_2}.$  A similar competing behavior between the $L^{1}$ and $L^{2}$ norms was already visible in Equation~\eqref{eq:ClaimP} which leads to
\begin{align}\label{norma1}
\lim_{\delta\to 0}\E\left[\widehat{C}_{d-1}^{(\delta, T_N)}(u)\right]&= p_{X(\0)}(u)\E\left[\|\nabla X(\0)\|_{1}|X(\0)=u\right],
\end{align}
with $\widehat{C}_{d-1}^{(\delta, T_N)}(u)$ as in~\eqref{eq:C1tildeNEW}. In addition, we see from the last equality in~\eqref{def:Pp} that
\begin{align*}
\E[\widetilde{C}^T_{d-1}(u)] &= \frac{1}{\sigma_d(T)}\int_{T}  \int_{-\infty}^\infty  \E\big[\delta_0\big(x-u\big)||\nabla X(t)||_1\ \big|\ X(t) = x\big]p_{X(\0)}(x)\,  \rmd x\,\rmd t\\
&= p_{X(\0)}(u)\E\left[\|\nabla X(\0)\|_{1}|X(\0)=u\right].
\end{align*}

Equation~\eqref{norma1}  should be put in parallel with the desired limit given by Rice's formula (see, \emph{e.g.}, Equation (6.27) in \cite{AW09}  or
Proposition 2.2.1 in \cite{berzin2017kac} (with $j=1$) written for a stationary process), \emph{i.e.},
\begin{align}\label{norma2}
\E[C^{T}_{d-1}(u)] &= p_{X(\0)}(u)\E[\|\nabla X(\0)\|_{2}|X(\0)=u],
\end{align}
with $C^{T}_{d-1}(u)$ as in~\eqref{eq:UestCT1bis}.
This  difference of norms in the limits in~\eqref{norma1} and~\eqref{norma2} motivates  the presence of the dimensional constant $\beta_d$ relating the $L_{1}$ and $L_{2}$ norms (see main Theorem~\ref{CLTjoint} and  Equation~\eqref{betad}).\\

Taking inspiration from the relationship between~\eqref{eq:UestCT1bis}  and~\eqref{eqn:P2_crofton}, the following proposition provides a similar result for $\widetilde{C}^T_{d-1}$ in~\eqref{def:Pp}  (first item) and a $L^1$ control between the estimator $\hat C^{(\delta,T)}_{d-1}$ and  $ \widetilde{C}^T_{d-1}$   (second item) as $\delta\to0$. To this end, we introduce the following  technical assumption.
\begin{enumerate}[label={($\mathcal{A}$3)}]
\item\label{A3}
Fix $j\in\{1,\ldots,d\}$. Let $f(t):=(X(t), \partial_{j}X(t))$ for $t\in\R^d$, where $\partial_{j}X$  denotes the partial derivative of $X$ in the direction $\be_j$. Fix $u\in\R$, and let $U\subset\R^2$ be a neighbourhood containing the point $(u,0)$. Suppose that the density of $f(\0)$ is bounded uniformly on $U$. 
Define
$$W := \sup_{t\in[0,1]^d}\max_{1\le k\le d}\big\{|\partial_{k}X(t)|, |\partial_{k}\partial_{j}X(t)|\big\}$$
and suppose that $\E[W^2\ |\ f(\0) = s] < \infty$ for all $s\in U$.

Similarly, suppose that for some interval $I\subset \R$ containing $u$, the marginal density of $X(\0)$  is bounded uniformly on $I$, 
and that $\E[\sup_{t\in[0,1]^d}||\nabla X(t)||_2\ |\ X(\0)=s]<\infty$,  for all $s\in I$.
\end{enumerate}

This assumption is central in the proof of the key Lemma~\ref{Lemma2} in Section~\ref{prf:secSA}. It allows for the use of techniques similar to the ones used in  \cite{Leadbetter} (see the proof of Proposition~\ref{LeadbetterTheorem} in Section~\ref{prf:secSA}).  Lemma~\ref{Lemma2} is used in the proof of the following proposition.

\begin{Proposition}\label{prp:Phat_L2_P1}
Let $X$ be a  random field on $\R^d$  as in Assumption~\ref{A0}.  It holds that
\begin{itemize}
  \item[i)]$\widetilde{C}^T_{d-1}(u)=\frac{1}{\sigma_d(T)}  \sum_{j=1}^d \int_{\vect(\be_j^\perp)} \sigma_0(L^T_X(u) \cap l_{\be_j, \mathbf{v}})\ \rmd \mathbf{v}$,  \, a.s.
\item[ii)] If, furthermore, $X$  satisfies Assumptions~\ref{A1} and~\ref{A2} for $\mathbf{w}=\be_{j}$, for each $1\le j\le d$, and for some $\eps\in (0,1)$, and~\ref{A3}, then there exists a constant $K$ such that for all $\delta \in (0,1)$
\begin{equation}\label{eq:Papproxcontrol}
\E\Big[\Big|\hat C^{(\delta,T)}_{d-1}(u)-\widetilde{C}^T_{d-1}(u) \Big|\Big]\le K \delta^{1-\eps}.
\end{equation}
\end{itemize}
\end{Proposition}

The proof of Proposition~\ref{prp:Phat_L2_P1}  in postponed to Section~\ref{proofProp31}.
Remark that isotropy is not required in Proposition~\ref{prp:Phat_L2_P1}.

\subsection{Strong alpha mixing random fields}

We present and discuss sufficient hypotheses
to prove the asymptotic normality in Theorem~\ref{CLTjoint}, below. We impose some  spatial asymptotic independence conditions for the   random field $X$ which also apply to integrals of  continuous functions over the level-curves of $X$ (such as the surface area in~\eqref{eq:UestCT1bis}). To this end,  mixing conditions are particularly appropriate (see for instance \citet{cabana1987affine}, \citet{Iribarren}).

\begin{Definition}[Strongly mixing random field]\label{mixingfield}
Let $X:= \{X(t): t \in \R^d\}$ be a random field satisfying Assumption~\ref{A0}, and let  $\sigma_U:=\sigma\{X(t): t \in U\}$ for a subset $U \subset \R^d$, \emph{i.e.}  the $\sigma-$field generated by $\{X(t): t \in U\}$. We define the following mixing coefficient for Borelian subsets $U, V \subset \R^d$,
\begin{align}\label{mixing}
\alpha(U, V) := \sup_{A \in \sigma_U, \,  B \in \sigma_V}\{|\P(A \cap B)-  \P(A)\P(B)|\}.
\end{align}
Further we define
$$\alpha(s) := \sup\{\alpha(U, V) : d_2(U, V)>s\},$$
where $d_2(U,V) :=\inf\{||u-v||_2:u\in U, v\in V\}$. A random field $X$ is said to be strongly mixing   if $\alpha(s) \rightarrow 0$ for $s\rightarrow \infty$.
\end{Definition}
Existing results that establish asymptotic normality of geometric quantities classically rely on the con\-ti\-nu\-ous observation of $X$ on $T$ and on a quasi-association  notion of dependence  (see \cite{alexander2007limit}, \cite{MESCHENMOSER}, \cite{bulinski2012central}, \cite{Spodarev13}) or are specific to Gaussian random fields (see \cite{meschenmoser2013functional}, \cite{Mu16},  \cite{Kratz2018}). There are also results for fields observed on the fixed lattice grid $\Z^{d}$ in \cite{RVY}, where the notion of clustering spin model is introduced and is implied by either mixing assumptions or quasi-association. Finally, it is worth noting that \cite{bulinski2010central} proposes a limit theorem for the empirical mean of $X$ observed on a grid with vanishing size and large $T$, \textit{i.e.} the same grid of observation $\mathbb{G}(\delta,T)$ as in the present work.

Corollary 3 in \cite{dedecker1998central}  (see also \cite{Bolthausen}) proposes minimal  $\alpha$-mixing conditions  to get  central limit results for stationary random fields observed on the fixed lattice $\Z^{d}$. However, we can not rely on such results as we are interested in non-trivial functionals of the field and we aim at imposing the conditions on the underlying field $X$ and not on the observed sequence on the varying lattice $\delta\Z^{d}.$ Instead,  to establish Theorem~\ref{CLTjoint} below, we heavily rely on the inheritance properties of mixing sequences: if a random field $X$ is strongly alpha-mixing (see Definition  \ref{mixingfield}), so is any measurable transformation of it. The latter property does not hold for quasi-association properties. Examples of mixing random fields include Gaussian random fields (see \cite{doukhan1994mixing} Section 2.1 and Corollary 2 for an explicit control of~\eqref{mixing}), and therefore any transformation of  Gaussian random fields such as Student or chi-square fields, or  Max-infinitely divisible random fields  (see \cite{dombry2012strong}). We refer the reader to  \cite{bradley2005basic} for other examples of processes satisfying various mixing conditions.

\subsection{Joint Central Limit Theorem}

In the following, $\stackrel{\mathcal{L}}\longrightarrow$ denotes the convergence in law.
Taking advantage of the results of Section~\ref{sec:BiasSA} and adapting the results of  \citet{Iribarren}, we state the following joint limit result for the estimated volume and  estimated  surface area.

\begin{Theorem}[Joint central limit theorem for  $\widehat{C}_d^{(\delta, T_N)}$ and $\widehat{C}_{d-1}^{(\delta, T_N)}$]\label{CLTjoint}
Let $X$ be an isotropic random field satisfying Assumptions~\ref{A0},~\ref{A1},~\ref{A2} for some $\eps\in(0,1)$,  and~\ref{A3}. Assume that  $X$   is  strongly mixing as in Definition~\ref{mixingfield} and that for some $\eta >0$, the mixing coefficients satisfy the rate condition
\begin{align*}
 \sum_{r=1}^{+\infty} r^{3 d-1} \alpha(r)^{\frac{\eta}{2+\eta}} < + \infty,
\end{align*}
and  $\E[\sigma_{d-1}\big(L_X(u) \cap [0,1]^{d}\big)^{2+\eta}] < + \infty$.
Let $(T_N)_{N\geq 1}$ be a sequence of hypercubes in $\R^d$ such that $\sigma_d(T_N)= (2 N \delta)^{d}$.
Define $$\widehat{C}^{(\delta, T_N)}(u) := \left(\widehat{C}_d^{(\delta, T_N)}(u), \widehat{C}_{d-1}^{(\delta, T_N)}(u)\right)^t  \quad  \mbox{ and }\quad C^*(u):=\left(C_d^{*}(u),  \frac{2 d}{\beta_d} C_{d-1}^{*}(u) \right)^t,$$
with $\widehat{C}_d$ (resp. $\widehat{C}_{d-1}$) as in~\eqref{eq:C2pixel} (resp. in~\eqref{eq:C1tildeNEW}) on the hypercubic lattice $\mathbb{G}(\delta, T)$ in~\eqref{G}, $C_k^{*}(u)$ as in~\eqref{defCast}, where  $(\cdot)^t$ denotes the matrix transposition and $\beta_{d}$ is as in~\eqref{betad}.  Then, there exists a finite covariance matrix $\Sigma(u)$  such that, if  $\Sigma_{(i,j)}(u) >0$, $\forall \, 1\leq i,j \leq 2$, it holds
\begin{align*} 
\sqrt{\sigma_d(T_N)} \left(\widehat{C}^{(\delta, T_N)}(u) - C^*(u)\right) & \stackrel{\mathcal{L}}\longrightarrow \mathcal{N}_{2}\big(0,\Sigma(u)\big),\end{align*}
as $\delta\rightarrow 0$, $N \delta\rightarrow\infty$, and   $(N\delta)^{d}\delta^{1-\eps}\to 0$, with $\eps$ as in Assumption~\ref{A2}, with variances
\begin{align*} 
\Sigma_{(1,1)}(u)&= \int_{\R^d} \Cov(\I{X(\0) \geq u}, \I{X(t)  \geq u}) \rmd t, \\
\Sigma_{(2,2)}(u)&= \int_{\R^d} \Big[p_{X(\0),X(t)}(u,u)\E\big[||\nabla X(\0)||_1 ||\nabla X(t)||_1\ \big|\ X(\0) = X(t) = u\big]\\
&\qquad- \Big(p_{X(\0)}(u)\E\big[||\nabla X(\0)||_1\ \big|\ X(\0) = u\big]\Big)^2 \Big] \rmd t.
\end{align*}
and covariances
$\Sigma_{(1,2)}(u)= \lim\limits_{N \rightarrow\infty} \sigma_d(T_N) \Cov(C^{T_N}_d(u), C^{T_N}_{d-1}(u))$.
\end{Theorem}

\begin{proof}[Proof of Theorem~\ref{CLTjoint}]
Let $C^{T_N}(u):=\left(C^{T_N}_d(u), \widetilde{C}^{T_N}_{d-1}(u)\right)^t$, with $C^{T_N}_d(u)$ as in~\eqref{eq:UestCT2} and $\widetilde{C}^{T_N}_{d-1}(u)$ as in~\eqref{def:Pp}. We decompose
\begin{align*}
\sqrt{\sigma_d(T_N)} \Big(\widehat{C}^{(\delta, T_N)}(u)& - C^*(u)\Big)
=  \sqrt{\sigma_d(T_N)} \left(\widehat{C}^{(\delta, T_N)}(u) - C^{T_N}(u)\right)   +
\sqrt{\sigma_d(T_N)} \left(C^{T_N}(u) - \E[C^{T_N}(u)]\right) \\
&+  \sqrt{\sigma_d(T_N)} \left(\E[C^{T_N}(u)] - \E[\widehat{C}^{(\delta, T_N)}(u)])\right) + \sqrt{\sigma_d(T_N)} \left(\E[\widehat{C}^{(\delta, T_N)}(u)] - C^*(u)\right)\\ &\hspace{1.5cm}:={I_1}+{I_2}+{I_3}+{I_4}.
\end{align*} 
From Proposition~\ref{prop:AireI1}, we get that first coordinates of $I_{1}$ and  $I_3$ go in probability to zero. The second coordinates are handled with    Proposition~\ref{prp:Phat_L2_P1}. It follows that $I_1 \stackrel{\PP}\longrightarrow 0$ and $I_3 \stackrel{\PP}\longrightarrow 0$, for  $N \delta\rightarrow\infty$, $\delta\rightarrow 0$ and $(N\delta)^{d}\delta^{1-\eps}\to 0$ as $N\to\infty$, with $\eps$ as in Assumption~\ref{A2}.\\
The joint central limit theorem under mixing conditions  (see  Theorem~\ref{CLTjointContinous},   for $n=N\delta$) for the random vector  $C^{T_N}(u)$ gives that  $I_2\stackrel{\mathcal{L}}\longrightarrow \mathcal{N}_{2}\big(0,\Sigma(u)\big)$. The given asymptotic variances  come  from \cite{cotsakis2022perimeter} (Equation~(11)), Equation~\eqref{def:Pp} and  Equation~\eqref{CovarianceBulinski}.
Notice that    $\widehat{C}_d^{(\delta, T)}(u)$ in~\eqref{eq:C2pixel} does not generate  any bias in the estimation of  the volume measure $\sigma_{d}$ (see~\eqref{areadensity}), so the first coordinate of $I_4$ is trivially equal to zero. The second coordinate of  $I_4$ is obtained via Corollary~\ref{cor:sqr}.  Then, by  Slutsky's Theorem,  we obtain the  result.
\end{proof}

All auxiliary results necessary for the proof of the  joint central limit theorem for  $(\widehat{C}_d^{(\delta, T_N)}, \widehat{C}_{d-1}^{(\delta, T_N)})$  are provided in Section~\ref{proofCTLSection}.

\begin{Remark}
The regime restriction $(N\delta)^{d}\delta^{1-\eps}\to 0$ (where $\eps$ can be small)  is imposed by Proposition~\ref{prp:Phat_L2_P1} and Proposition~\ref{prop:AireI1}. It can be  improved by requiring more stringent assumptions: \emph{e.g.}, if $X$ is quasi-associated (see \cite{alexander2007limit} and \cite{bulinski2012central}) and under a decay assumption for its correlation function (see  Appendix~\ref{QuasiAssociativitySection}). However, improving the rate in Proposition~\ref{prp:Phat_L2_P1} seems much more delicate and is out of the scope of the present work.

Generalizing Theorem~\ref{CLTjoint} to point clouds other than hypercubic lattices requires that one first identifies the  associated random variable when $\delta$ gets small, \textit{i.e.} an analog to~\eqref{def:Pp}, and that one proves an analog of~\eqref{eq:Papproxcontrol}.  In this sense, this asymptotic result is lattice-dependent and generalizing it to general tessellations is also an interesting open point.
\end{Remark}

 \section{Additional results and proofs associated to Section~\ref{sec:BiasSA}\label{prf:secSA}}

 \subsection{Auxiliary lemmas on crossings}\label{sectionLemma}

The following technical lemma controls the probability that a one-dimensional random process crosses the level $u$ more than once on a small interval of length $t$.

\begin{Lemma}[Level crossings of random processes]\label{lem:1D_small_sojourn}
Let $Y = \{Y(s): s\in\R\}$ be a one-dimensional stationary random process with twice continuously differentiable sample paths. Suppose that the probability density function of $Y(0)\big|\{Y'(0)=0\}$ is uniformly bounded by $M < \infty$ and that $\|Y''\|_{\infty,[0,1]}\in L^p$ for some $p> \frac12$. For $u\in\R$, let $Y^{-1}(u) := \{s\in\R:Y(s) = u\}$. Then, there exists a constant $K\in\R^+$ such that
\begin{equation*} 
\P\big({\rm Card}\left(Y^{-1}(u)\cap[0,t]\right)\ge 2\big) \leq K t^{\frac{3p}{p+1}},
\end{equation*}
for all $t\in[0,1]$.
\end{Lemma}

\begin{proof}[Proof of Lemma~\ref{lem:1D_small_sojourn}]
If there are to exist two points $s_1,s_2\in [0,t],\ s_1<s_2$ such that $Y(s_1)=Y(s_2)=u$, then by Rolle's theorem, there exists a $c\in[s_1,s_2]$ such that $Y'(c)=0$. Moreover, by Taylor's theorem,
$$Y(s_2) = u = Y(c) + \frac{(s_2 - c)^2}{2}Y''(\eta),$$
for some $\eta\in[c,s_2]$, which implies
$2|Y(c) - u| \leq t^2\|Y''\|_{\infty,[0,1]}$,  where as $Y$ is twice continuously differentiable, $ \|Y''\|_{\infty,[0,1]}<\infty$ almost surely. It implies that for any $x\in(0,1)$,
$$\{{\rm Card}\left(Y^{-1}(u)\cap[0,t]\right)\ge 2\}\subseteq \{\|Y''\|_{\infty,[0,1]}\geq (2/t^2)^x\} \cup \{\exists c\in[0,t] : Y'(c)=0, |Y(c)-u|\leq (t^2/2)^{1-x}\}.$$
To control the probability of the first event, Markov's inequality gives
$$\P\big(\|Y''\|_{\infty,[0,1]}\geq (2/t^2)^x\big) = \P\big(\|Y''\|_{\infty,[0,1]}^p\geq (2/t^2)^{px}\big)\leq \Big(\frac{t^2}2\Big)^{px}\E[\|Y''\|_{\infty,[0,1]}^p].$$
As for the second event, we write
\begin{align*}
\P\big(\exists c\in[0,t] :& Y'(c)=0, |Y(c)-u|\leq (t^2/2)^{1-x}\big)\\
&= \P\big({\rm Card}\left(\{c\in[0,t] : Y'(c)=0\}\right)\ge1\big)
\P\big(|Y(s)-u|\leq (t^2/2)^{1-x}\ |\ Y'(s)=0, s\in[0,t]\big)\\
&\leq \E\big[\mathrm{Card}\left(\{c\in[0,t] : Y'(c)=0\}\right)\big]\P\big(|Y(s)-u|\leq (t^2/2)^{1-x}\ |\ Y'(s)=0\big).
\end{align*}
By Theorem 3.3.1 in \cite{buetler1966}, $\E[\mbox{Card}(\{c\in[0,t] : Y'(c)=0\})] = t\E[\mbox{Card}(\{c\in[0,1] : Y'(c)=0\})]$. In addition, $\P\big(|Y(s)-u|\leq (t^2/2)^{1-x}\ |\ Y'(s)=0\big) \leq 2M(t^2/2)^{1-x}$ as $Y$ is stationary and $ Y(0)\big|\{Y'(0)=0\}$ has bounded density. In total, we have shown that
$$\P\big({\rm Card}\left(Y^{-1}(u)\cap[0,t]\right)\ge 2\big) \leq \Big(\frac{t^2}2\Big)^{px}\E[\|Y''\|_{\infty,[0,1]}^p] + 2M\E[\mbox{Card}(\{c\in[0,1] : Y'(c)=0\})]t\Big(\frac{t^2}{2}\Big)^{1-x}.$$
Thus, optimizing in $x$ leads to consider $x:= 3/(2p+2)\in(0,1)$, if $p>\frac12$, and to the result.
\end{proof}

The following lemma provides a bound on the crossing probability for general, possibly anisotropic,  random fields.

\begin{Lemma}\label{lem:crossing_bound}
Let $X$ be a random field on $\R^d$ satisfying~\ref{A0}. Fix $t\in\R^d$ and $u\in \R$. Then,
$$\P\left(X(\0) \leq u < X(t)\right ) \leq C_{d-1}^*(u)\|t\|_2.$$
\end{Lemma}
\begin{proof}[Proof of Lemma~\ref{lem:crossing_bound}]
Let $\bar{t}$ denote the line segment in $\R^d$ whose endpoints are $\0$ and $t$. Define $K_X(u) := \{L_X(u) - \lambda t : \lambda\in(0,1)\}$ to be the set of points that are at a distance of at most $||t||_2$ from $L_X(u)$ in the direction of $-t$ (see Figure~\ref{fig:kxu} for an illustration  of the set $K_X(u)$ in  $\R^2$). By the assumption of stationarity it holds that
\begin{align*}
\P\left(X(\0) \leq u < X(t)\right )&\leq \P\left( \bar{t} \cap L_X(u) \neq \emptyset\right )= \P\left( \0 \in K_X(u)\right)
 = \E\left[\frac{\sigma_d\left(K_X(u)\cap B_r^d\right)}{\sigma_d(B_r^d)}\right],
\end{align*}
for all $r\in\R^+$. If we denote $r^+ := r + ||t||_2$ for $r\in\R^+$, then
\begin{align*}
\sigma_d\left(K_X(u)\cap B_r^d\right) &\leq \int_{L_X(u) \cap B_{r^+}^d}\big\langle t,\, \frac{\nabla X(s)}{||\nabla X(s)||_2}\big\rangle\,\sigma_{d-1}(\rmd s)\leq \int_{L_X(u) \cap B_{r^+}^d}||t||_2\,\sigma_{d-1}(\rmd s)\\
&=\|t\|_2 \sigma_{d-1}\left(L_X(u)\cap B_{r^+}^d\right), \quad a.s.
\end{align*}
Taking the expectation and sending $r\to\infty$ yields the result, since $\sigma_d(B_r^d)/ \sigma_d(B_{r^+}^d) \to 1$.
\end{proof}

\begin{figure}[H]
  \centering
\includegraphics[width=7.5cm]{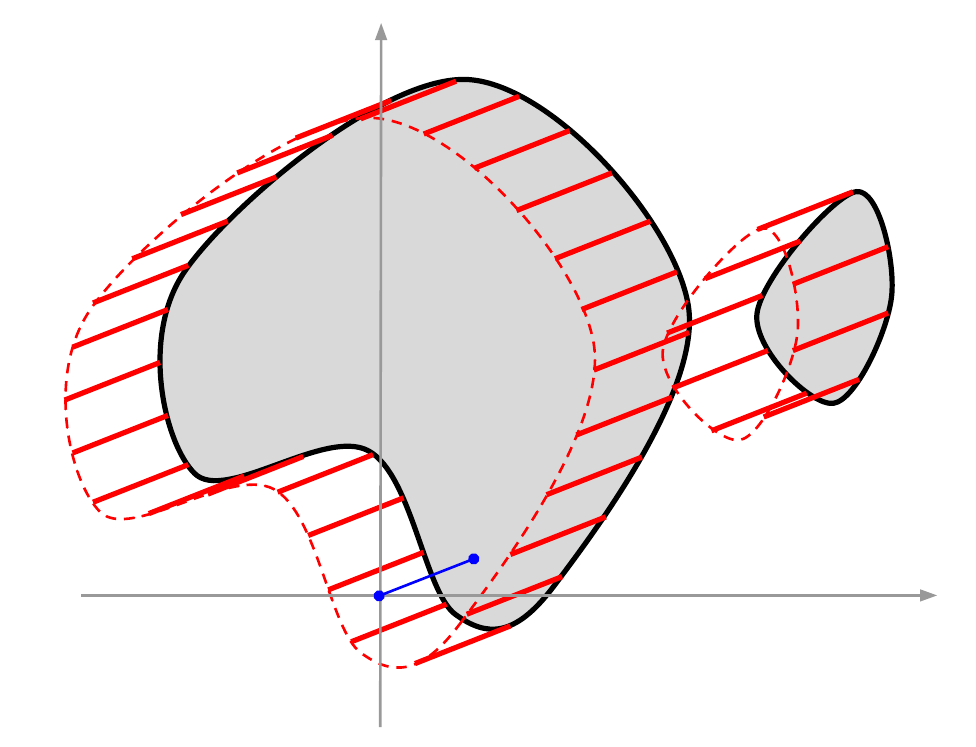}
    \put(-116,42){$\bar{t}$}
\caption{In dimension 2, for the level curve $L_{X}(u)$ in black,  $K_X(u)$ is the area  covered by solid red line segments. Note that the line segment $\bar{t}$ (in blue) crosses $L_X(u)$ (in black) if and only if $\0\in K_X(u)$.\label{fig:kxu}}
\end{figure}

\begin{Definition}\label{def:Pij}
Define the map $\pi_j:\R^d\rightarrow\R^d$ to be the orthogonal projection onto the $(d-1)$-dimensional subspace $\vect(\mathbf{e}_j^\perp)$. That is, for $S\subseteq \R^d$,
$$\pi_j(S) = \Big\{\sum_{i\ne j}s_{i}\be_{i}:\mathbf{s}\in S\Big\} = \{\mathbf{v}\in\vect(\mathbf{e}_j^\perp): l_{\mathbf{e}_j,\mathbf{v}}\cap S \neq \emptyset\},$$
with $l_{\mathbf{e}_j,\mathbf{v}}$ as  in~\eqref{eqlines}.
\end{Definition}

The following lemma allows to obtain the rate of convergence of a Riemann sum used in Proposition~\ref{prp:Phat_L2_P1}. The techniques used in the proof can be used to bound the probability that a random manifold intersects a small region of space.

\begin{Lemma}\label{Lemma2}
 Let $X$ be  a  stationary  random field on $\R^d$ satisfying Assumption~\ref{A0} and~\ref{A3}.
Fix $j\in\{1,\ldots,d\}$.
For $\mathbf{v}\in\vect(\be_j^\perp)$, denote $n_\mathbf{v} := \sigma_0(L_X(u) \cap l_{\be_j,\mathbf{v}} \cap [0,1]^d)$, with $l_{\be_j,\mathbf{v}}$ as  in~\eqref{eqlines}.
 Then, there exists a constant $K\in\R^+$ such that for all $\delta\in[0,1)$,
$$\P\bigg( \sup_{\mathbf{v}_1, \mathbf{v}_2 \in \pi_j([0,\delta]^d)} \big|n_{\mathbf{v}_1} - n_{\mathbf{v}_2}\big| > 0\bigg) \leq K \delta.$$
\end{Lemma}

\begin{proof}[Proof of Lemma~\ref{Lemma2}]
Note that $n_\mathbf{v}$ seen as a function that maps $\mathbf{v}\in\vect(\be_j^\perp)$ into $\N_0$ is almost surely piecewise constant on $\pi_j([0,1]^d)$. The discontinuities occur at points in $\pi_j(H_1\cup H_2 \cup H_3)$, where $H_1 := f^{-1}((u,0))\cap [0,1]^d$---recall from Assumption~\ref{A3} that $f = (X,\partial_j X)$, $H_2 :=\{\mathbf{s}\in L_X(u):s_j = 0\}$, and $H_3 :=\{\mathbf{s}\in L_X(u):s_j = 1\}$ (see Figure~\ref{fig:pw_constant} for an example with $d=2$).

\begin{figure}[H]
  \centering
\includegraphics[width=9cm]{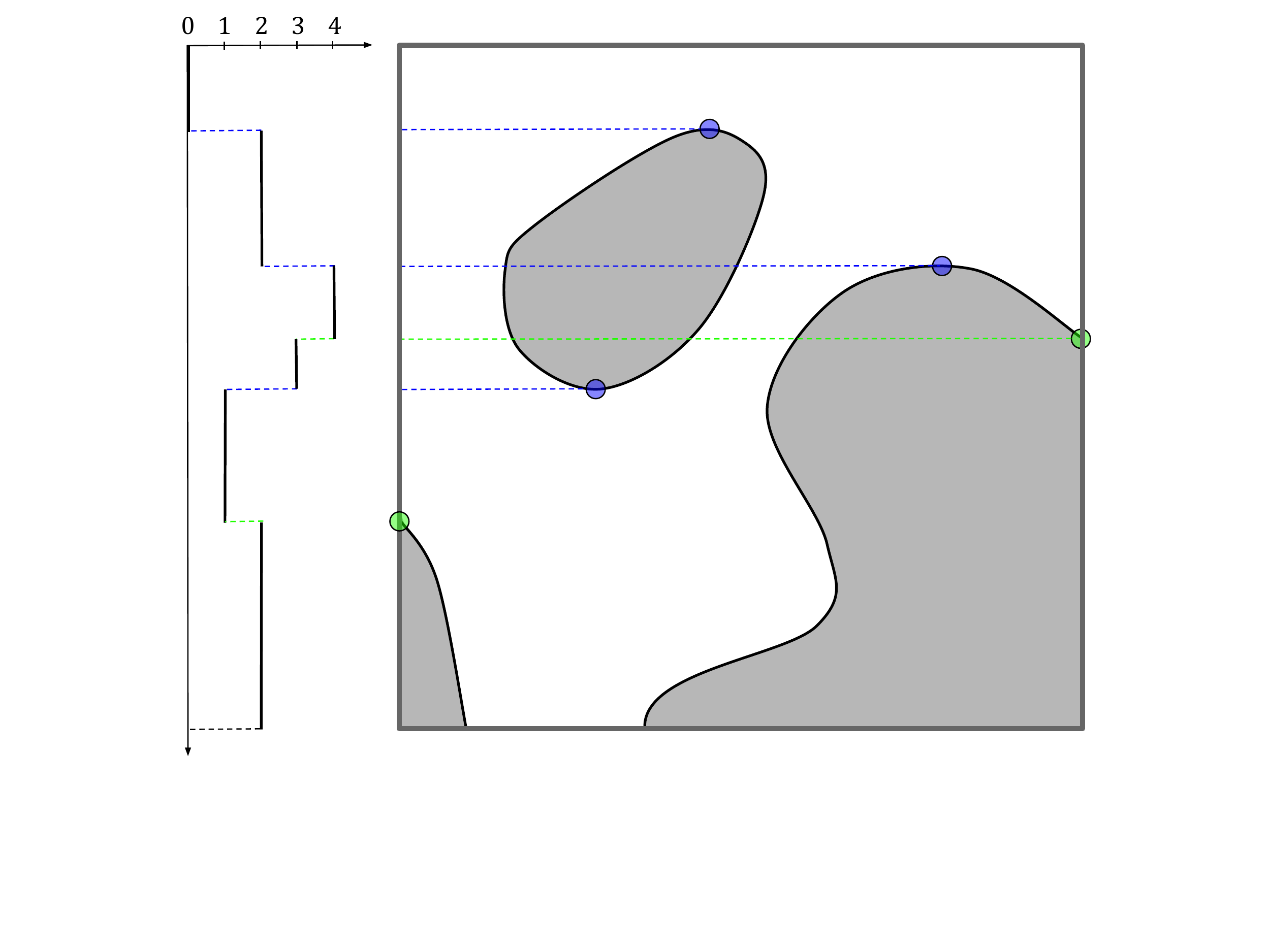}
\vspace{-1.2cm}
\caption{On the left a representation of the piecewise constant function $n_{\mathbf{v}}$ of $\mathbf{v}\in\vect(\be_1^\perp)$ computed for the excursion set represented in grey on the right. The points in blue are in $H_1$ and the points in green are in $H_2$ and $H_3$.\label{fig:pw_constant}}
\end{figure}

For $i\in\{1,2,3\}$, we aim to show that $\P\big(\pi_j([0,\delta]^d)\cap\pi_j(H_i)\neq \emptyset\big)$ is of the order of $\delta$, which implies the desired result.

We begin with the case of $H_1$. By stationarity, we apply the union bound,
\begin{equation}\label{eqn:stationarity_argument_H1}
\P\big(\pi_j([0,\delta]^d)\cap\pi_j(H_1)\neq\emptyset\big) \leq \lceil 1/\delta\rceil \P\big([0,\delta]^d\cap H_1\neq\emptyset\big),
\end{equation}
where $\lceil\cdot\rceil$ denotes the ceiling function.\\

By Taylor's theorem, for each $t\in[0,\delta]^d$, there exists an $s\in[0,\delta]^d$ such that
$f(t) = f(\0) + \langle t, \nabla f(s)\rangle.$
If $[0,\delta]^d \cap H_1\neq \emptyset$, then
$(u,0) - f(\0) = \langle t, \nabla f(s)\rangle,$
for some $s, t\in[0,\delta]^d$. \\

Let us denote by $M_1\in\R^+$ the constant which  uniformly  bounds the density of $f(\0)$ on $U$, and $M_2$ the constant which  bounds  the marginal density of $X(\0)$ on $I$ (see Assumption \ref{A3}).

For $\delta \leq 1$, we bound the variation of $f(\0)$ by writing $$||(u,0) - f(\0)||_\infty \leq ||t||_1 W \leq d\delta W.$$
Therefore, we obtain
\begin{align*}
\P\big([0,\delta]^d\cap H_1\neq\emptyset\big) \leq \P\big(||(u,0) - f(\0)||_\infty \leq d\delta W\big).
\end{align*}
let $g(w,x,\dot{x})$ be the joint probability distribution of $W$, $X(\0)$, and $\partial_j X(\0)$. Note that
\begin{align*}
\P\big(||(u,0) - f(\0)||_\infty \leq d\delta W\big) &= \int_{0}^\infty \int_{u-\delta wd}^{u+\delta wd}\int_{-\delta wd}^{\delta wd} g(w,x,\dot{x})\, \rmd\dot{x}\rmd x\rmd w\\
&= (2d\delta)^2 \int_{0}^\infty \int_{-w/2}^{w/2}\int_{-w/2}^{w/2} g(w,u +2d\delta z, 2d\delta \dot{z})\, \rmd\dot{z}\rmd z\rmd w,
\end{align*}
where a change of the variables $x =u + 2d\delta z$ and $\dot{x} = 2d\delta \dot{z}$ was used. Now, by the dominated convergence theorem,
\begin{align*}
\lim_{\delta\to 0} \frac{1}{\delta^2}\P\big(||(u,0) - f(\0)||_\infty \leq d\delta W\big) &= \lim_{\delta\to 0} 4d^2 \int_{0}^\infty \int_{-w/2}^{w/2}\int_{-w/2}^{w/2} g(w,u + 2d\delta z, 2d\delta \dot{z})\, \rmd\dot{z}\rmd z\rmd w\\
& = 4d^2 \int_{0}^\infty w^2 g(w,u,0)\, \rmd w \leq 4d^2 M_1\E[W^2\ |\ f(\0) = (u,0)],
\end{align*}
which is finite by~\ref{A3}. Thus,
$\P\big([0,\delta]^d\cap H_1\neq\emptyset\big) = O(\delta^2),$
and by~\eqref{eqn:stationarity_argument_H1},
$\P\big(\pi_j([0,\delta]^d)\cap\pi_j(H_1)\neq\emptyset\big) = O(\delta),$
as desired.

Now, we consider the case of $H_2$ (the case of $H_3$ being identical). Using similar arguments as in the case of $H_1$, we readily obtain
\begin{align*}
\lim_{\delta\to 0} \frac{1}{\delta}\P\big(\pi_j([0,\delta]^d)\cap\pi_j(H_2)\neq\emptyset\big) &= \lim_{\delta\to 0} \frac{1}{\delta}\P\big([0,\delta]^d\cap H_2\neq\emptyset\big)\\
&\leq \lim_{\delta\to 0} \frac{1}{\delta}\P\Big(|u-X(\0)| \leq \sqrt{d} \delta \sup_{t\in[0,1]^d}||\nabla X(t)||_2\Big)\\ 
&\leq 2\sqrt{d} M_2 \E\Big[\sup_{t\in[0,1]^d}||\nabla X(t)||_2\ \big|\ X(\0) = u\Big] < \infty.
\end{align*}
Thus, $\P\big(\pi_j([0,\delta]^d)\cap\pi_j(H_2)\neq\emptyset\big) = O(\delta)$ as desired. 
\end{proof}

 \subsection{Proof of Proposition~\ref{LeadbetterTheorem}}
First note that $$\left\{X(\0)\leq u<X({q}\be_{1})\right\}=\left\{X(\0)\leq u<X(\0)+qX_q\right\}=\left\{X(\0)\leq u\right\}\cap\left\{\frac{1}{q}(u-X(\0))<X_{q}\right\}.$$
It follows that
\begin{align*}
\frac{1}{q}\PP\left(X(\0)\leq u<X(q)\right)&=\frac{1}{q}\PP\left(X(\0)\leq u,\ X_{q}>q^{-1}(u-X(\0))\right)\\
&=\frac1{q}\int_{-\infty}^{u}\int_{\frac1{q}(u-x)}^{\infty}g_{q}(x,y)\rmd y\rmd x=\frac1{q}\int_{0}^{\infty}\int_{u-qy}^{u}g_{q}(x,y)\rmd x\rmd y\\
&=\int_{0}^{\infty}z\int_{0}^{1}g_{q}(u-qzv,z)\rmd v\rmd z,
\end{align*} using the change of variable $v=\frac{u-x}{qy},\ z=y $ ($x=u-qzv$, $y=z$).  From the assumptions, we derive that $g_{q}(u-qzv,z)\to p(u,z)$ pointwise as $q\to0$, and the dominated {convergence} theorem applies and leads to the result.

 \section{Additional results and proofs associated to Section~\ref{sec:TCL}\label{prf:secTCL}}

\subsection{Proof of Proposition~\ref{prp:Phat_L2_P1}}\label{proofProp31}

Firstly, we prove the first item in Proposition~\ref{prp:Phat_L2_P1}.  As $L^T_X(u)$ is a $C^1$ smooth manifold, integrals over $L^T_X(u)$ can be expressed via sequences of polygonal approximations. Let $(A_1, \ldots, A_k)$ be a finite sequence of $(d-1)$-dimensional hyperplanes residing in $\R^d$. For $i\in\{1,\ldots,k\}$, let $\mathbf{n}_i\in\R^d$ denote the unit normal vector to $A_i$. We aim to show
\begin{equation}\label{eqn:c_tilde_riemann_sum}
\sum_{i=1}^k\sum_{j=1}^d |\langle\mathbf{n}_i, \be_j\rangle|\sigma_{d-1}(A_i) = \sum_{j=1}^d \int_{\vect(\be_j^\perp)} \sigma_0\bigg(\Big(\bigcup_{i=1}^k A_i\Big) \cap l_{\be_j, \mathbf{v}}\bigg)\ \rmd \mathbf{v},
\end{equation}
with $l_{\be_j,\mathbf{v}}$ as  in~\eqref{eqlines}. By noticing that for all $\mathbf{v}\in\vect(\be_j^\perp)$,
$$\sigma_0\bigg(\Big(\bigcup_{i=1}^k A_i\Big) \cap l_{\be_j, \mathbf{v}}\bigg) = \sum_{i=1}^k\I{A_i\cap l_{\be_j, \mathbf{v}}\neq \emptyset},$$
we see that it suffices to show that for each pair of indices $(i,j)\in \{1,\ldots,k\}\times \{1,\ldots,d\}$,
\begin{equation}\label{eqn:projected_area}
|\langle\mathbf{n}_i, \be_j\rangle|\sigma_{d-1}(A_i) = \int_{\vect(\be_j^\perp)} \I{A_i\cap l_{\be_j, \mathbf{v}}\neq \emptyset}\ \rmd \mathbf{v}.
\end{equation}
Equation~\eqref{eqn:projected_area} is evident, since both sides of the equality describe the $\sigma_{d-1}$ measure of the projection of $A_i$ onto the hyperplane $\vect(\be_j^{\perp})$. Note that $\sum_{j=1}^d |\langle\mathbf{n}_i, \be_j\rangle| = ||\textbf{n}_i||_1$, and that for $t\in L_X(u)$, $\nabla X(t)/||\nabla X(t)||_2$ is the unit normal vector to $L_X(u)$ at $t$, so the  left-hand side of Equation~\eqref{eqn:c_tilde_riemann_sum} is a Riemann sum that approximates the integral expression for $\widetilde{C}^T_{d-1}(u)$ in~\eqref{def:Pp} when $k$ goes to infinity.

Now, we prove the second item in Proposition~\ref{prp:Phat_L2_P1}.
For $\mathbf{g}\in \mathbb{G}(\delta, T)$ in~\eqref{G}, let $Z_{\mathbf{g}} :=\mathbbm{1}_{\{X(\mathbf{g})\ge u\}}$.
Equations~\eqref{eq:C2pixel}-\eqref{eq:C1tildeNEW}  can be rewritten as
\begin{align}
\widehat{C}_d^{(\delta, T)}(u)&=\frac{\delta^{d}}{\sigma_d(T)}\sum_{\mathbf{g}\in \mathbb{G}(\delta, T)}Z_{\mathbf{g}}=\frac{1}{\left(2N\right)^{d}}\sum_{\mathbf{g}\in \mathbb{G}(\delta, T)}Z_{{\mathbf{g}}},\nonumber 
\\
\widehat{C}_{d-1}^{(\delta, T)}(u)& = \frac{ \delta^{d-1}}{\sigma_d(T)} \sum_{j=1}^d\underset{\mathbf{g}+\delta\be_j\in\mathbb{G}(\delta,T)}{\sum_{\mathbf{g}\in\mathbb{G}(\delta,T)}} \mathbbm{1}_{\{Z_{\mathbf{g}} \neq Z_{\mathbf{g}+\delta{\be_j}}\}}.\nonumber 
 \end{align}
Set $G:=\mathbb G(\delta, T) \cap [0,1]^d$ and $n_\mathbf{v} := \sigma_0(L_X(u) \cap l_{\be_j,\mathbf{v}} \cap [0,1]^d)$, suppose that we  have established
\begin{equation}\label{eqn:to_show_CLT}
\E\Big[\Big|\delta^{d-1}\sum_{\mathbf{g}\in G}
\mathbbm{1}_{\{Z_{\mathbf{g}}\neq  Z_{\mathbf{g}+\delta{\be_{j}}}\}} - \int_{\vect{(\be_{j}^\perp})\cap[0,1]^{d}}n_\mathbf{v} \rmd \mathbf{v} \Big|\Big] = O\big(\delta^{1-\eps}\big),
\end{equation}
for all $j\in\{1,\ldots,d\}$.
Then, by stationarity and the triangle inequality,
$$\E\Big[\Big|\delta^{d-1}\sum_{\mathbf{g}\in\mathbb G(\delta, T)} \mathbbm{1}_{\{Z_{\mathbf{g}}\neq  Z_{\mathbf{g}+\delta{\be_{j}}}\}} - \int_{\vect(\be_j^\perp)}\sigma_0\big(L_X^T(u) \cap l_{\be_j,\mathbf{v}}\big) \rmd \mathbf{v} \Big|\Big] = O\Big((N\delta)^{d}\delta^{{1-\eps}}\Big)$$
 so the desired result holds by summing over the $d$ dimensions and again applying the triangle inequality together with the first item. Therefore, it suffices to show~\eqref{eqn:to_show_CLT}.
This is done in two steps. First, we show that (recall  Definition~\ref{def:Pij})
\begin{equation}\label{eq:step1_CLT}
\E\Big[\Big|\delta^{d-1}\sum_{\mathbf{g}\in G} \mathbbm{1}_{\{Z_{\mathbf{g}}\neq  Z_{\mathbf{g}+\delta{\be_{j}}}\}} - \delta^{d-1}\sum_{\mathbf{v}\in\pi_j(G)}n_{\mathbf{v}}\Big|\Big] = O(\delta^{1-\eps}).
\end{equation}
Second, we show that
\begin{equation}\label{eq:step2_CLT}
\E\Big[\Big|\delta^{d-1}\sum_{\mathbf{v}\in\pi_j(G)}n_{\mathbf{v}} - \int_{\vect(\be_j^\perp)\cap[0,1]^{d}}n_{\mathbf{s}} \rmd \mathbf{s} \Big|\Big] = O(\delta^{1-\eps}).
\end{equation}

Let $M^{(\delta)} := \big(\sigma_0(G)\big)^{1/d}=\lfloor \delta^{-1}\rfloor+1$ be the number of rows in the grid $G$ that is contained in $[0,1]^d$. By construction, $M^{(\delta)}\delta\le 1+\delta$ for all $\delta\in\R^+$.

To see that~\eqref{eq:step1_CLT} holds, we use the triangle inequality to write
\begin{align}\label{eq:sum_expectations_CLT}
\E\Big[\Big|\delta^{d-1}&\sum_{\mathbf{g}\in G} \mathbbm{1}_{\{Z_{\mathbf{g}}\neq  Z_{\mathbf{g}+\delta{\be_{j}}}\}} - \delta^{d-1}\sum_{\mathbf{v}\in\pi_j(G)}n_{\mathbf{v}}\Big|\Big] \le \delta^{d-1}\sum_{\mathbf{v}\in\pi_j(G)}\E\Big[\Big|\sum_{\mathbf{g}\in G\cap l_{\be_j,\mathbf{v}}} \mathbbm{1}_{\{Z_{\mathbf{g}}\neq  Z_{\mathbf{g}+\delta{\be_{j}}}\}} - n_{\mathbf{v}}\Big|\Big].
\end{align}
Moreover, for fixed $\mathbf{v}\in \pi_j(G)$, $\sum_{\mathbf{g}\in G\cap l_{\be_j,\mathbf{v}}} \mathbbm{1}_{\{Z_{\mathbf{g}}\neq  Z_{\mathbf{g}+\delta{\be_{j}}}\}}$ approaches $n_{\mathbf{v}}$ from below, and both quantities take values in $\N_0$. If the two quantities are distinct, then there must be two elements of $L_X(u) \cap [0,1]^d \cap l_{\be_j,\mathbf{v}}$ with a spacing of less than $\delta$. Let $A_1^{(\mathbf{v})}$ denote the event
$\big\{\sum_{\mathbf{g}\in G\cap l_{\be_j,\mathbf{v}}} \mathbbm{1}_{\{Z_{\mathbf{g}}\neq  Z_{\mathbf{g}+\delta{\be_{j}}}\}} \neq n_{\mathbf{v}}\big\}$.
With $(\cdot)_j$ denoting the projection onto the $j^{\mathrm{th}}$ component, note that
$$A_1^{(\mathbf{v})} \subset \bigcup_{i=0}^{M^{(\delta)}-1}\bigg\{\mbox{Card}\Big([i\delta,(i+2)\delta]\cap \big(L_X(u) \cap l_{\be_j,\mathbf{v}}\big)_j\Big)\geq 2\bigg\},$$
since any two points in $\R$ with a spacing of less than $\delta$ will be contained in an interval $[i\delta,(i+2)\delta]$ for some $i\in\Z$. It follows from Assumption~\ref{A1}, Lemma~\ref{lem:1D_small_sojourn} with $p=2$, and the definition of $M^{(\delta)}$ that $\PP(A_1^{(\mathbf{v})})=O(\delta).$
This fact, with H\"older's inequality yields for the $\eps$ defined in~\ref{A2},
\begin{align*}
\E\Big[\Big|&\sum_{\mathbf{g}\in G\cap l_{\be_j,\mathbf{v}}} \I{Z_{\mathbf{g}}\neq Z_{\mathbf{g}+\delta\be_j}} - n_{\mathbf{v}}\Big|\Big] =
\E\Big[\Big|\sum_{\mathbf{g}\in G\cap l_{\be_j,\mathbf{v}}} \I{Z_{\mathbf{g}}\neq Z_{\mathbf{g}+\delta\be_j}} - n_{\mathbf{v}}\Big|\mathbbm{1}_{A_1^{(\mathbf{v})}}\Big]\\
&\leq \E[\ n_{\mathbf{v}}\ \mathbbm{1}_{A_1^{(\mathbf{v})}}\ ] \leq \big(\E[n_\mathbf{v}^{1/\eps}]\big)^{\eps}\ \P(A_1^{(\mathbf{v})})^{1-\eps} = O(\delta^{1-\eps}).
\end{align*}
  By the stationarity of $X$, each of the  $(M^{(\delta)})^{d-1}$  terms in~\eqref{eq:sum_expectations_CLT} is identical, and as  $M^{(\delta)}\delta\le 1+\delta$,~\eqref{eq:step1_CLT} follows immediately.  We continue by showing~\eqref{eq:step2_CLT}. For $\mathbf{v}\in\vect(\be_j^{\perp})$, define the event
$A_2^{(\mathbf{v})} := \Big\{ \sup_{\mathbf{s}\in\mathbf{v}+\pi_j([0,\delta]^d)}
|n_\mathbf{v}-n_\mathbf{s}| > 0 \Big\}.$
Then, the triangle inequality    gives
\begin{align*}
\E&\Big[\Big|\delta^{d-1}
 \sum_{\mathbf{v}\in\pi_j(G)}
n_{\mathbf{v}} - \int_{\vect(\be_j^\perp)\cap[0,1]^{d}}n_{\mathbf{s}} \rmd \mathbf{s} \Big|\Big]
\leq \delta^{d-1} \sum_{\mathbf{v}\in\pi_j(G)}
\E\Big[\Big|n_\mathbf{v} - \delta^{1-d}\int_{\mathbf{v}+\pi_j([0,\delta]^d)} n_\mathbf{s}\ \rmd \mathbf{s}\Big|\Big]\\
& = \delta^{d-1} \sum_{\mathbf{v}\in\pi_j(G)}
\E\Big[\Big|n_\mathbf{v} - \delta^{1-d}\int_{\mathbf{v}+\pi_j([0,\delta]^d)} n_\mathbf{s}\ \rmd \mathbf{s}\Big| \mathbbm{1}_{A_2^{(\mathbf{v})}}\Big]\\
&\leq \delta^{d-1} \sum_{\mathbf{v}\in\pi_j(G)}
\Big(\E\Big[\Big|n_\mathbf{v} - \delta^{1-d}\int_{\mathbf{v}+\pi_j([0,\delta]^d)} n_\mathbf{s}\ \rmd \mathbf{s}\Big|^{\frac 1 \eps}\Big]\Big)^{ \eps }\P(A_2^{(\mathbf{v})})^{1-\eps }\\
&\leq \delta^{d-1} \sum_{\mathbf{v}\in\pi_j(G)} \bigg(\big(\E\big[ n_\0^{\frac 1 \eps}\big]\big)^{ \eps } + \Big(\E\Big[\Big|\delta^{1-d}\int_{\mathbf{v}+\pi_j([0,\delta]^d)} n_\mathbf{s}\ \rmd \mathbf{s}\Big|^{\frac 1 \eps}\Big]\Big)^{ \eps }\bigg)\P(A_2^{(\mathbf{v})})^{{1-\eps} }\\
& \leq \delta^{d-1} \sum_{\mathbf{v}\in\pi_j(G)} \bigg(\big(\E\big[n_\0^{\frac 1 \eps}\big]\big)^{ \eps } + \Big(\delta^{1-d}\int_{\mathbf{v}+\pi_j([0,\delta]^d)} \E\big[n_\mathbf{s}^{\frac 1 \eps}\big]\ \rmd \mathbf{s}\Big)^{\eps }\bigg)\P(A_2^{(\mathbf{v})})^{{1-\eps} }\\
&\leq \delta^{d-1} \sum_{\mathbf{v}\in\pi_j(G)} 2\big(\E\big[n_\0^{\frac 1 \eps}\big]\big)^{ \eps }\P(A_2^{(\mathbf{v})})^{{1-\eps} },
\end{align*}
where we used Jensen inequality for the penultimate inequality and the fact that $\sigma_{d-1}(\pi_{j}([0,\delta]^{d}))=\delta^{d-1}$ for the last inequality.
Of the $(M^{(\delta)})^{d-1}$ terms in the sum over $\pi_j(G)$, there are $({M^{(\delta)}}-1)^{d-1}$ terms such that $\mathbf{v}+\pi_j([0,\delta]^d)\subset \pi_j([0,1]^d)$. For these terms, by stationarity, $\P(A_2^{(\mathbf{v})})=\P(A_2^{(\mathbf{0})}) \leq K_1\delta$ for some $K_1\in\R^+$ by Lemma~\ref{Lemma2}. For the remaining ${M^{(\delta)}}^{d-1} - ({M^{(\delta)}}-1)^{d-1}$ terms in the sum, the bound $\P(A_2^{(\mathbf{v})}) \leq 1$ suffices. Thus,
\begin{align*}
\E\Big[\Big|\delta^{d-1}
&\sum_{\mathbf{v}\in\pi_j(G)}
n_{\mathbf{v}} - \int_{\vect(\be_j^\perp)\cap[0,1]^{d}}n_{\mathbf{s}} \rmd \mathbf{s} \Big|\Big]\\
&\leq 2\delta^{d-1}\Big(({M^{(\delta)}}-1)^{d-1}(K_1\delta)^{{1-\eps} } + \big({M^{(\delta)}}^{d-1} - ({M^{(\delta)}}-1)^{d-1}\big)\Big)\big(\E\big[\big|n_\mathbf{0}\big|^{\frac 1 \eps}\big]\big)^{ \eps } = O\big(\delta^{ {1-\eps} }\big)
\end{align*}
 and~\eqref{eq:step2_CLT} holds which completes the proof.

\subsection{Auxiliary results for the proof of the joint central limit theorem}\label{proofCTLSection}
 
Analogously to \eqref{eq:Papproxcontrol}, the following result provides an $L^1$ control of the approximation error of the first coordinate of $I_{1}$ (\emph{i.e.}, the estimated volume) in the proof of Theorem~\ref{CLTjoint}. To this end, notice that we cannot directly apply an approximation inequality such as Proposition 4 in  \cite{bierme:hal-02793752} as the function $y\mapsto \mathbbm{1}_{\{y\ge u\}}$ appearing in  the definition of the volume in~\eqref{eq:UestCT2} is not Lipschitz.

\begin{Proposition}\label{prop:AireI1}
Let $X$ be an isotropic random field  satisfying  Assumption~\ref{A0}. Let  $C^T_d(u)$ as in~\eqref{eq:UestCT2} and $\widehat{C}_d^{(\delta, T)}(u)$   as in~\eqref{eq:C2pixel}. It holds that 
\begin{align}\label{eq:RI1}
\E\left[\left |\widehat{C}_d^{(\delta, T)}(u)  - C^T_d(u) \right|\right]\le K {\delta},
\end{align} 
where $K\in\R^+$  is a constant independent of $\delta$ and $\sigma_d(T)$.
\end{Proposition}

\begin{Remark}\rm\label{improved_{area}}\rm
Proposition~\ref{prop:AireI1} provides a control of the discretization error in the computation of the volume. This error is of the same order as   in Proposition~\ref{prp:Phat_L2_P1} (item \emph{ii}). This control can be greatly improved by asking more stringent assumptions: \emph{e.g.}, if $X$ is quasi-associated (see \cite{alexander2007limit} and \cite{bulinski2012central}) and under a decay assumption for its correlation function.  In this  setting, the bound in~\eqref{eq:RI1} can be improved to get an upper bound in $L^{2}$ with a faster rate in $\delta (\sigma_{d}(T))^{-1}$ instead of $\delta$. The interested reader is referred to the Appendix \ref{QuasiAssociativitySection}.
\end{Remark}

\begin{proof}[Proof of Proposition~\ref{prop:AireI1}]
Remark that
\begin{align*}
\left |\widehat{C}_d^{(\delta, T)}(u)  - C^T_d(u) \right|
&= \left | \frac{\delta^d}{\sigma_d(T)} \sum_{t \in \mathbb{G}(\delta, T)} \left ( \I{X(t)\geq u} - \frac{1}{\delta^d} \int_{t + [0,\delta)^d} \I{X(s)\geq u}\, \rmd s\right ) \right |\\
&= \left | \frac{1}{\sigma_d(T)} \sum_{t \in \mathbb{G}(\delta, T)} \int_{t + [0,\delta)^d}\left ( \I{X(t)\geq u} -  \I{X(s)\geq u}\right )\, \rmd s \right |\\
&\leq \frac{1}{\sigma_d(T)} \sum_{t \in \mathbb{G}(\delta, T)} \int_{t + [0,\delta)^d}\left | \I{X(t)\geq u} -  \I{X(s)\geq u} \right |\, \rmd s.
\end{align*}
Therefore, by Lemma \ref{lem:crossing_bound}, it holds that
\begin{align*}
\E\left[\left |\widehat{C}_d^{(\delta, T)}(u)  - C^T_d(u) \right|\right]
&\leq \frac{1}{\sigma_d(T)} \sum_{t \in \mathbb{G}(\delta, T)} \int_{t + [0,\delta)^d} \Big(\P\left(X(t) < u \leq X(s)\right) + \P\left(X(t) \geq u > X(s)\right)\Big)\, \rmd s\\
&\leq \frac{1}{\sigma_d(T)} \sum_{t \in \mathbb{G}(\delta, T)} \int_{t + [0,\delta)^d} {2 C_{d-1}^*(u)}\|t-s\|_2 \, \rmd s
\leq {2 \sqrt{d}C_{d-1}^*(u)}\delta.
\end{align*}
\end{proof}

Corollary~\ref{cor:sqr} below describes the behaviour  of the second coordinate of the term  $I_4$  in our main Theorem~\ref{CLTjoint}, by using  the second order approximation in Theorem~\ref{prp:Pcrossing_Esurfacearea} for the  specific hypercubic   point-referenced $d$-honeycomb $\delta\dot{\mathcal{G}}$ in Section~\ref{hypercubic:section}.

\begin{Corollary}\label{cor:sqr}
Consider a discrete regular hypercubic  grid $\mathbb{G}(\delta, T)$   as in ~\eqref{G}.   Let $X$ be an isotropic random field satisfying Assumptions~\ref{A0},~\ref{A1} and~\ref{A2} for some $\eps\in(0,1)$.
Let $C_{d-1}^*(u)$ be as in~\eqref{defCast} and   $\widehat{C}_{d-1}^{(\delta, T_N)}(u)$ as  in~\eqref{eq:C1tildeNEW}. Then, it holds that $$\sqrt{\sigma_d(T_N)} \left(\E[\widehat{C}_{d-1}^{(\delta, T_N)}(u)] -   \frac{2 d}{\beta_d} C_{d-1}^{*}(u) \right)  \rightarrow  0$$
for $N \delta \rightarrow\infty$,  $ (N\delta)^{d/2} \delta^{1-\eps}   \rightarrow 0$   with  $\beta_d$ as in~\eqref{betad} and $\eps$ as in Assumption~\ref{A2}.
\end{Corollary}

\begin{proof}  It is a direct consequence of~\eqref{eq:C1tildeNEW}  the stationarity and isotropy of $X$, Equation \eqref{eqn:second_item_thm2.1} in  Theorem~\ref{prp:Pcrossing_Esurfacearea}, and the fact that by construction Card$\left(\mathbb G(\delta, T)\right)=(2N)^{d}$ and $\sigma_d(T_N)= (2 N \delta)^{d}$.
\end{proof}
 
Finally, we focus on the term $I_2$ in the proof of Theorem~\ref{CLTjoint}. The following theorem establishes a  joint central limit theorem for the  vector  $\left(C^{T_n}_d(u), \widetilde{C}^{T_n}_{d-1}(u)\right)$ in the case where  $T_n $ is   a sequence of hypercubes in $\R^d$  such that  $T_n \nearrow \R^d$ as $n \rightarrow\infty$. Our result is based on techniques used in \cite{Iribarren}.

\begin{Theorem}\label{CLTjointContinous}
Let $u$ be a fixed level in $\R$.
Assume that  $X$   is  strongly mixing as in Definition~\ref{mixingfield} such that for some $\eta >0$, the mixing satisfies the rate condition
\begin{align*}
 \sum_{r=1}^{+\infty} r^{3 d-1} \alpha(r)^{\frac{\eta}{2+\eta}} < + \infty,
\end{align*}
and  $\E[\sigma_{d-1}\big(L_X(u) \cap [0,1]^{d}\big)^{2+\eta}] < + \infty$. Let $C^{T_n}(u):=\left(C^{T_n}_d(u), \widetilde{C}^{T_n}_{d-1}(u)\right)^t$.
Let $(T_n)_{n\geq 1}$ be a sequence of hypercubes in $\R^d$ such that $\sigma_d(T_n)= (2 n)^{d}$.
Then there exists  a finite covariance matrix $\Sigma(u)$ such that, if  $\Sigma(u)>0$,
\begin{align*}
\sqrt{\sigma_d(T_n)} \left(C^{T_n}(u) - \E[C^{T_n}(u)]\right) & \stackrel{\mathcal{L}}\longrightarrow \mathcal{N}_{2}\big(0,\Sigma(u)\big),\end{align*} as $n\to\infty$.
\end{Theorem}

\begin{proof}
Let $\mathbf{t}=(t_1, t_2)\in\R^{2}$ and $V_\ii(1)$  be as in Section~\ref{hypercubic:section}.    Write
$$\xi_\mathbf{i} := t_1\int_{L_X(u) \cap V_\ii(1)} \frac{\|\nabla X(s)\|_1}{\|\nabla X(s)\|_2}\ \sigma_{d-1}(\rmd s) + t_{2}\sigma_d\big(E_X(u)\cap V_\ii(1)\big),$$
for $\mathbf{i}\in\Z^d$, which makes $\xi := \{\xi_\mathbf{i}:\mathbf{i}\in\Z^d\}$ a stationary random field on $\Z^d$. It is straightforward to see that $\xi$ inherits the mixing property of the random field $X$ and that the mixing coefficients of $\xi$ satisfy $\alpha_\xi(n+\lceil\sqrt{d}\rceil) \leq \alpha_X(n)$ for all $n\in\N^+$, since the diameter of $V_\ii(1)$ is $\sqrt{d}$. Notice also that $|\xi  _\mathbf{i} | \leq ||\mathbf{t}||_2 \big(\sqrt{d}\sigma_{d-1}\big(L_X(u) \cap V_\ii(1)\big) + \sigma_d\big(E_X(u)\cap V_\ii(1)\big)\big) \leq  ||\mathbf{t}||_2 \big(\sqrt{d}\sigma_{d-1}\big(L_X(u) \cap V_\ii(1)\big) +  1\big) $ almost surely, so  the $2+\eta$ moment of $\xi_\mathbf{i}$ is finite.
Finally, since
\begin{equation*}
\langle\mathbf{t}, C^{T_n}(u) \rangle = \sum_{\mathbf{i}\in\Z^d} \xi_\mathbf{i}\,\I{V_\ii(1)\,\subset \,T_n},
\end{equation*}
the proof is completed with an application of Proposition 1 and Lemma 1 in \citet{Iribarren}  and the Cram{\'e}r-Wold device.
\end{proof}

\bigskip
\textbf{Acknowledgments}:    This work has been supported by the French government, through the 3IA C\^{o}te d'Azur Investments in the Future project managed by the National Research Agency (ANR) with the reference number ANR-19-P3IA-0002. This work has been partially supported  by the project  ANR MISTIC (ANR-19-CE40-0005).

\bibliographystyle{apalike}
\bibliography{biblio.bib}

\bigskip

\begin{appendix}

\section{Supplementary material}\label{sec:annexe}

\subsection{Proposition~\ref{prop:AireI1} improved under the assumption of quasi-associativity}\label{QuasiAssociativitySection}

The obtained bound   in Proposition~\ref{prop:AireI1} comes without any assumption on the correlation function of the field $X$.  Indeed the proof Proposition~\ref{prop:AireI1} only focuses on the control of the pixelisation error in a fixed domain. The growing domain $T_N$ is considered \emph{a posteriori} without taking into account the asymptotic spatial independency between  pixels with respect to their Euclidean  distance. Obviously the bound in~\eqref{eq:Papproxcontrol} only  depends   on $\delta$ and not $\sigma_{d}(T)$.

In this supplementary material we aim to improve the  bound   in Proposition~\ref{prop:AireI1}
by adding dependency conditions on the field $X$ and considering a $L^{2}$ control.    The proposed $L^2$ control in Proposition~\ref{prop:AireRDimensionD} below can be used to transfer for instance a multilevel central limit theorem provided for  ${C}_d^{T}(u)$  in~\eqref{eq:UestCT2} to $\widehat{C}_d^{(\delta,T)}(u)$  in~\eqref{eq:C2pixel} (see  Corollary~\ref{CLT} below) under to the sole assumptions  $N \delta\to \infty$ and   $\delta \to  0$. To improve the result of Proposition~\ref{prop:AireI1} we will use that  distant pixels have a small correlation.

\begin{Proposition}\label{prop:AireRDimensionD}
Let $X$  a  random field  satisfying  Assumption~\ref{A0} with $\V(X(\0))=\upsilon^2$, $\upsilon > 0$.  Suppose moreover that  $X(\0)$ has density bounded by $M$ and that
\begin{enumerate}[label={($\mathcal{H}$1)}]\item\label{A1Supp}  $X$ is quasi-associated, i.e.,
\[| {\rm Cov}(f(X_{I}), g(X_{J}))\le |{\rm Lip}(f){\rm Lip}(g)|\sum_{s\in I}\sum_{t\in J}|{\rm Cov}(X(s),X(t))|,\]
\end{enumerate}
for all disjoint finite sets $I, J \subset T$ and any Lipschitz functions $f:\R^{I}\mapsto \R$ and $g:\R^{J}\mapsto \R$. \newline Assume the following decay condition for the correlation function  $\rho(t) = \corr(X(\0), X(t)),\ t\in\R^{d}$:
\begin{enumerate}[label={($\mathcal{H}$2)}]
\item\label{A2Supp}  $|\rho(t)|\leq (1+\|t\|)^{-\gamma}$, $\gamma> {3d}$.
\end{enumerate}
Then, it holds that
\begin{align}\label{eq:VdR}
\V\left({C}_d^{T}(u)   - \widehat{C}_d^{(\delta,T)}(u)\right)\le \kappa \frac{\delta}{\sigma_d(T)},
\end{align}
where $\kappa$  is a constant only depending on  $d,\,M,\,\upsilon^{2}, \, C^{*}_{d-1}(u)$ in \eqref{defCast} and  $\gamma$.
\end{Proposition}

\begin{Remark}[Quasi-associativity property]\rm
Assumption~\ref{A1Supp}  is called  \textit{quasi-associativity property} and is classically used  to get the asymptotic normality of  ${C}^{T}_d(u)$ in~\eqref{eq:UestCT2} (see for instance  \cite{bulinski2012central}, \citet{Spodarev13}).
\cite{bulinski1998asymptotical} shows that positively associated random fields are quasi-associated,
\cite{shashkin2002quasi} ensures that Gaussian fields are quasi-associated and \cite{alexander2007limit} that shot noise fields are also quasi-associated (see Theorem 1.3.8).   Notice that if  $X$ is a  stationary Gaussian  random field on $\R^d$,   then   \cite{bulinski2012central}  shows that Assumption~\ref{A2Supp} on the covariance function can be softened as  $\gamma> d$. The interested reader is referred to Lemma 2 and Equation (34) in \cite{bulinski2012central}  together with Lemma~\ref{LemmaCOV},  postponed and proved  below. \end{Remark}

\begin{Remark}\rm\label{rem:growing_domain}\rm
\begin{sloppypar} The control~\eqref{eq:VdR} on the variance provides information about the amplitude of the error made when replacing the uncomputable quantity ${C}^{T}_d(u)$   in~\eqref{eq:UestCT2} with its accessible version $\widehat{C}_d^{(\delta,T)}(u)$  in~\eqref{eq:C2pixel}. Two regimes can be considered for this variance to be small:   \textit{i)} $T_{N}\nearrow \R^{d}$, \textit{i.e.} $N\delta\to \infty$, and $\delta\to \ell\ge 0$ or \textit{ii)} $\delta \rightarrow 0$ with the domain $T$ held constant. Both situations are of interest, but if one wishes to transfer a central limit result for growing domain $T$ from      the true volume ${C}_d^{T}(u)$ to $\widehat{C}_d^{(\delta,T)}(u)$   one should require both $T_{N}\nearrow \R^d$ and $\delta \to 0$ (see Corollary~\ref{CLT}   below).  \end{sloppypar}
\end{Remark}

\begin{proof}[Proof of Proposition~\ref{prop:AireRDimensionD}]
Let $\textbf{i} \in \Z^d$.  Recall that the reference point  of $\delta \dot{\mathcal{G}}$ (see Section~\eqref{hypercubic:section})  is denoted $t_{\ii, \delta} :=(\delta i_1, \ldots, \delta i_d)$. For $\ii$ and $\jj$ in $\big[|-N,N-1|\big]^{d}$, $\ii \ne \jj$ stands for $\exists \, k\in\{1,\ldots, d\}$, $i_{k}\ne j_{k}.$ Similarly, $\ii	> \jj$ (resp. $\ii	< \jj$)  means that $\forall\,   k\in\{1,\ldots, d\}$ $i_{k}>j_{k}$ (resp. $i_{k}<j_{k}$), $\ii =  \ii'$ means $i_1=i'_1, \ldots, i_d = i'_d$ and  $\sum_{\ii=m}^{M}$ stands for $\sum_{i_{1}=m}^{M}\ldots\sum_{i_{d}=m}^{M}.$
We study the remainder term $R_{\delta}:=\sigma_{d}(T)\left(\hat C^{(\delta,T)}_{d}(u)-C^T_{d}(u)\right)$ given by
\begin{align*}
R_{\delta}=\int_{T}\mathbbm{1}_{\{X(t)\ge u\}}\rmd t-\sum_{t \in \mathbb{G}(\delta, T)}\mathbbm{1}_{\{X(t)\ge u\}}\delta^{{d}}&=\sum_{\ii =-N}^{N-1}  \int_{V_{\ii}(\delta)}\Big(\mathbbm{1}_{\{X(t)\ge u\}}-\mathbbm{1}_{\{X(t_{\ii, \delta})\ge u\}}\Big)\rmd t,
\end{align*}
from Equation~\eqref{eq:C2pixel}.  Denote $d_{\ii}(t):=\mathbbm{1}_{\{X(t)\ge u\}}-\mathbbm{1}_{\{X(t_{\ii, \delta})\ge u\}}$,  for $t \in V_{\ii}(\delta)$. Note that by stationarity this term is centered and $\V(R_{\delta})=\E[R_{\delta}^{2}].$
  We  write,
\begin{align*}
\E[R_{\delta}^{2}]&=\sum_{\ii, \ii' =-N}^{N-1} \E\left[\int_{V_{\ii}(\delta)}\int_{V_{\ii'}(\delta)}d_{\ii}(t)d_{\ii'}(t')\rmd t\,\rmd t'\right],
\end{align*}
where $d_{\ii}(t):=\mathbbm{1}_{\{X(t)\ge u\}}-\mathbbm{1}_{\{X(t_{\ii, \delta})\ge u\}}$,  for $t \in V_{\ii}(\delta)$.
We now study the quantities
\begin{align*}
\E[d_{\ii}(t)d_{\ii'}(t')]&=\E\big[(\mathbbm{1}_{\{X(t)\ge u\}}-\mathbbm{1}_{\{X(t_{\ii, \delta})\ge u\}})(\mathbbm{1}_{\{X(t')\ge u\}}-\mathbbm{1}_{\{X(t_{\ii', \delta})\ge u\}})\big] \\
&=\Cov(\mathbbm{1}_{\{X(t-t')\ge u\}},\mathbbm{1}_{\{X(\0)\ge u\}})+\Cov(\mathbbm{1}_{\{X(t_{\ii, \delta}-t_{\ii', \delta})\ge u\}},\mathbbm{1}_{\{X(\0)\ge u\}})\\
&\hspace{1cm}-\Cov(\mathbbm{1}_{\{X(t-t_{\ii', \delta})\ge u\}},\mathbbm{1}_{\{X(\0)\ge u\}})-\Cov(\mathbbm{1}_{\{X(t_{\ii, \delta}-t')\ge u\}},\mathbbm{1}_{\{X(\0)\ge u\}}),
\end{align*}
for $t\in V_{\ii}(\delta),\ t'\in V_{\ii'}(\delta)$.
To connect the later quantity to the covariance function of $X$ we rely on Lemma 7.3.5.  in \cite{alexander2007limit} (see Lemma~\ref{LemmaCOV} where a revised proof is provided for sake of clarity).  Then, we obtain,
\begin{align*}|\E[d_{\ii}(t)d_{\ii'}(t')]|&\le K_{M}\big(|r(t-t')|^{1/3}+|r(t_{\ii, \delta}-t_{\ii', \delta})|^{1/3}+|r(t-t_{\ii', \delta})|^{1/3}+|r(t'-t_{\ii, \delta})|^{1/3}).\end{align*}  
Under Assumption~\ref{A2Supp} and using the structure of the approximating grid, it follows that
\begin{align}
\label{eq:dij1}\E[d_{\ii}(t)d_{\ii'}(t')]&\le 4\,\sigma^{2/3}K_{M}
\big(1+\|t_{\ii', \delta}-t_{\ii + \1, \delta}\|\big)^{-\gamma/3}\quad \mbox{if } {\ii'>\ii}.
\end{align}

The latter bound gets small whenever ${\rm dist}(t_{\ii+\1, \delta},t_{\ii'+\1})$ is large. Hereafter, we propose another control of this term when this distance is small.
It holds that \begin{align*}\mbox{for } t\in V_{\ii}(\delta), \quad d_{\ii}(t)&= \begin{cases}
1 & \mbox{if}\ X(t_{\ii, \delta}) \leq u < X(t),\\
-1 & \mbox{if}\ X(t)\leq u < X(t_{\ii, \delta}),\\
0 & \mbox{otherwise}.
\end{cases}
\end{align*}
It follows from the hypothesis of stationarity that $\E[d_{\ii}(t)]\le \PP(|d_{\ii}(t)|=1)= \PP\big(X(t_{\ii, \delta}) \leq u < X(t)\big) + \PP\big(X(t) \leq u < X(t_{\ii, \delta})\big)$.
By Lemma~\ref{lem:crossing_bound} we obtain that
\begin{align*}
 \PP\big(X(\0) \leq u < X(t)\big)&\le C_{d-1}^*(u)\|t\|_2 \le  \sqrt{d}\delta C_{d-1}^*(u),
\end{align*}
Combining with  the definition of $d_{\ii}(t)$ and the Cauchy-Schwarz inequality provides
\begin{align}
\label{eq:dij2}\E[d_{\ii}(t)d_{\ii'}(t)]\le { \sqrt{\E[d_{\ii}(t)^{2}]\E[d_{\ii'}(t')^{2}]}}\le {4\sqrt{d}\delta}C_{d-1}^*(u).
\end{align}

Therefore, using~\ref{A0} we derive that
\begin{align*}
\V(R_{\delta})&=\left(\sum_{{\ii = \ii'}=-N}^{N-1}+ {2^{d}}\sum_{{\ii' \ge \ii = -N}\atop{{\ii \ne \ii'}}}^{N-1}\right)\int_{V_{\ii}(\delta)}\int_{V_{\ii'}(\delta)}\E\left[d_{\ii}(t)d_{\ii'}(s)\right]\rmd t\, \rmd s,
\end{align*}
Recall that $\sigma_d(T_{N})= (2 \delta N)^{d}$  and that $|d_{\ii}|\le 1$ \emph{a.s.}. Combining with ~\eqref{eq:dij1} and~\eqref{eq:dij2}, we obtain
\begin{align*}
\V(R_{\delta})&\le\sum_{\ii =-N}^{N-1}\delta^{ {2d}}+ 2^{d}\delta^{2d}
\sum_{{\ii' \ge \ii =-N}\atop{{\ii \ne \ii'}}}^{N-1}\frac{4\upsilon^{2/3}K_{M}}{\big(1+\|t_{\ii', \delta}-t_{\ii+\1, \delta}\|\big)^{\gamma/3}}\wedge \big\{{4\sqrt{d}\delta}C_{d-1}^*(u)\big\} \\
&=\delta^{d}\sigma_{d}(T_{N})+ \kappa \delta^{d}\frac{\sigma_{d}(T_{N})}{N^{d}}
\sum_{{\ii' \ge \ii =-N}\atop{{\ii \ne \ii'}}}^{N-1}\{(1+\|t_{\ii', \delta}-t_{\ii+\1, \delta}\|)^{-\gamma/3}\wedge \delta\},
\end{align*}
where $\kappa$  is a constant depending on $({d},M,\upsilon^{2},C^{*}_{d-1}(u))$.  As  for all $(\ii,\ii')$  it holds  {$\|t_{\ii', \delta}-t_{\ii, \delta}\|^{2}=\delta_{N}^{2}\sum_{k=1}^{d}(i'_{k}-i_{k})^{2}$},  we get
\begin{align*}
\frac{1}{N^{d}}&\sum_{{\ii' \ge \ii =-N}\atop{{\ii \ne \ii'}}}^{N-1}\left\{(1+\|t_{\ii', \delta}-t_{\ii+\1, \delta}\|)^{-\gamma/3}\wedge \delta_{N}\right\}
= \frac{1}{N^{d}}\sum_{{\ii' \ge \ii =-N}\atop{{\ii \ne \ii'}}}^{N-1}\left\{\left(1+\delta\sqrt{\sum_{k=1}^{d}(i'_{k}-i_{k}-1)^{2}}\right)^{-\gamma/3}\wedge\delta\right\}\\
&=\sum_{\ell_{1},\ldots,\ell_{d}=1}^{N-1}\left\{\left(1+\delta\sqrt{\sum_{k=1}^{d}\ell_{k}^{2}}\right)^{-\gamma/3}\wedge\delta\right\}.
\end{align*}
To bound this term, we use the following upper bound,
\[\left(1+\delta\sqrt{\sum_{k=1}^{d}\ell_{k}^{2}}\right)^{-\gamma/3}\wedge\delta\le \begin{cases}
\delta & \mbox{if }\sqrt{\sum_{k=1}^{d}\ell_{k}^{2}}\le \left\lfloor \frac{1}{\delta} \right\rfloor,\\
\left(1+\delta\sqrt{\sum_{k=1}^{d}\ell_{k}^{2}}\right)^{-\gamma/3}&\mbox{otherwise.}
\end{cases}\]
We derive that
\begin{align*}
\sum_{\ell_{1},\ldots,\ell_{d}=1}^{N-1}\left\{\left(1+\delta\sqrt{\sum_{k=1}^{d}\ell_{k}^{2}}\right)^{-\gamma/3}\wedge\delta\right\}\le& \delta\mbox{Card}\Big\{\ell_{1},\ldots,\ell_{d}\in\{1,\ldots,N-1\}^{d},\ \sqrt{\sum_{k=1}^{d}\ell_{k}^{2}}\le \Big\lfloor \frac{1}{\delta} \Big\rfloor\Big\}\\
&+\sum_{{\ell_{1},\ldots,\ell_{d}=1}\atop{ \sqrt{\sum_{k=1}^{d}\ell_{k}^{2}}>\lfloor 1/{\delta}\rfloor} }^{N-1}\left(1+\delta\sqrt{\sum_{k=1}^{d}\ell_{k}^{2}}\right)^{-\gamma/3}=:S_{1}+S_{2}.
\end{align*}
Furthermore, $S_{1}\le \delta\lfloor \delta^{-1}\rfloor^{d}\le \delta^{1-d}$ and
\begin{align*}
S_{2}&=\sum_{{\ell_{1},\ldots,\ell_{d}=1}\atop{ \sqrt{\sum_{k=1}^{d}\ell_{k}^{2}}>\lfloor 1/{\delta_{N}}\rfloor} }^{N-1}\left(1+\delta\sqrt{\sum_{k=1}^{d}\ell_{k}^{2}}\right)^{-\gamma/3}=\sum_{{p_{1},\ldots,p_{d}=0}\atop{ \|\mathbf{p}\|> {1} }}^{\lfloor(N-1)\delta\rfloor}(1+\|\mathbf{p}\|)^{-\gamma/3},
\end{align*}
where $S_{2}=O(1)$\footnote{Indeed, using that in $\R^{d}$ there exists $0<c_{d}< C_{d}$ such that $c_{d}\|p\|_{1}\le\|p\|_{2}\le C_{d}\|p\|_{1}$ we obtain \begin{align*}
\sum_{\mathbf{p}\ge0}\frac{1}{(1+\|\mathbf{p}\|)^{\alpha}}&\le \sum_{\mathbf{p}\ge0}\frac{1}{(1+c_{d}\|\mathbf{p}\|_{1})^{\alpha\sum_{r\ge 0}\quad\sum_{\mathbf{p}\ge 0, \|\mathbf{p}\|_{1}=r}\frac{1}{(1+r)^{\alpha}}}}=\sum_{r\ge 0}\frac{K_{d}(r)}{(1+r)^{\alpha}},
\end{align*} where $K_{d}(r):=\mbox{Card}\{\mathbf{p}\ge 0, \|\mathbf{p}\|_{1}=r\}$. An immediate upper bound is given by $K_{d}(r)\le r^{d-1}$, which lead to the sufficient condition $\alpha>d$ to get the convergence of the sum. Note that this condition should be necessary, to prove this one should use similar computations, $\|p\|_{2}\le C_{d}\|p\|_{1}$ and show (most delicate part) that $K_{d}(r)\ge c r^{d-1}$ for some $1>c>0.$} whenever $\gamma>{3d}$ if $N\delta\to \infty$ and for every $\gamma$ if $N\delta$ remains bounded.
Gathering all these elements, we obtain (assuming additionally that $\gamma> 3d$ if $N\delta\to \infty$),
\begin{align*}
\V(R_{\delta})&\le {\delta^{d}}{\sigma_d(T_{_{N}})}+ \,\kappa\,{\delta^{d}}\sigma_d(T_{N)})\big(\delta^{1-d}+S_{2}\big)=\kappa \, {\delta}{\sigma_d(T_{N})}+O\left({\delta^{d}}{\sigma_d(T_{N})}\right).
\end{align*}
 It follows that, $\V(R_{\delta})\le  \kappa \delta\sigma_d(T_N),$ where $\kappa$ is a constant depending on  $d,\,M,\,\upsilon^{2}, \, C^{*}_{d-1}(u)$ and $\gamma$ only. Dividing by $\sigma_{d}(T)^{2}$ completes the proof using that $\V(R_{\delta})=\sigma_{d}(T)^{2}\V\big(C_{d}^{T}(u)-\hat C_{d}^{(\delta,T)}(u)\big)$.
\end{proof}

\begin{Lemma}\label{LemmaCOV}
For a quasi-associated random vector $(U,V)$ such that $U$ and $V$ are square integrable and with densities bounded by $M>0.$ Then, for all $(u,v)\in \R$ it holds \begin{align*}
|{\Cov}(\mathbbm{1}_{U\ge u},\mathbbm{1}_{V\ge v})|\le 3\, a^{2/3}\, |\Cov(U,V)|^{1/3} := K_{M}|\Cov(U,V)|^{1/3}.
\end{align*}
\end{Lemma}
Lemma~\ref{LemmaCOV} is a  slight modification of  Lemma 7.3.5.  in \cite{alexander2007limit} and for sake of clarity,  we provide below a    proof.

\begin{proof}[Proof of Lemma~\ref{LemmaCOV}] If $\Cov(U,V)=0$, by Corollary 1.5.5  in \cite{alexander2007limit} implies that $U$ and $V$ are independent and the bound is trivial. Else for any $\delta >0$ and the function $h_{\delta,u}$ for $u\in\R$ such that
\[h_{\delta,u}(t)=\begin{cases}
0 & \mbox{if }t\le u-\delta,\\
\frac1\delta\big(t-(u-\delta)\big) & \mbox{if }t\in (u-\delta,u),\\
1 & \mbox{if }t\ge u.
\end{cases}\]
It follows from the triangle inequality that \begin{align*}
\big|\Cov(\mathbbm{1}_{U\ge u},\mathbbm{1}_{V\ge v})\big|
&\le \big|\Cov(h_{\delta,u}(U),h_{\delta,v}(V))\big|+\big|\Cov(\mathbbm{1}_{U\ge u}-h_{\delta,u}(U),h_{\delta,v}(V))\big|\\ &\quad\quad+\big|\Cov(\mathbbm{1}_{U\ge u},\mathbbm{1}_{V\ge v}-h_{\delta,v}(V))\big|\\
&\le \delta^{-2}|\Cov(U,V)|+4 M \delta,
\end{align*}   where the last inequality follows from Theorem 1.5.3 of \cite{alexander2007limit} to control the first term, the last two terms are controlled using that $\mathbbm{1}_{U\ge u}-h_{\delta,u}(U)$ is nonzero for $U\in(u,u-\delta)$ and that the density of $U$ is bounded by $M$.  Minimizing in $\delta$ this last term, we find   $\delta^{*}=\big(|\Cov(U,V)|/M\big)^{1/3}$. Then
\begin{align*}
\big|\Cov(\mathbbm{1}_{U\ge u},\mathbbm{1}_{V\ge v})\big|&\le \big(|\Cov(U,V)|/M\big)^{-2/3}|\Cov(U,V)|+4\,M\,\big(|\Cov(U,V)|/M\big)^{1/3}\\
&= 3\, M^{2/3}\, |\Cov(U,V)|^{1/3}.
\end{align*}
\end{proof}
Notice that in Lemma~\ref{LemmaCOV}, the factor 4 in Lemma 7.3.5.  in \cite{alexander2007limit}   is   improved to 2   as $-1\le  \mathbbm{1}_{U\ge u}-h_{\delta,u}(U)\le 1$ \emph{a.s}.\\

Proposition~\ref{prop:AireRDimensionD} can be used to transfer for instance multi-level central limit theorem provided for  ${C}_d^{T}(u)$  in~\eqref{eq:UestCT2} to $\widehat{C}_d^{(\delta,T)}(u)$  in~\eqref{eq:C2pixel} (see  Corollary~\ref{CLT} below).

\begin{Corollary}\label{CLT}\begin{sloppypar}
Let $X$  a  random field  satisfying assumptions in Proposition~\ref{prop:AireRDimensionD}. Let $\uu=(u_1, \ldots, u_m)^t \in \R^m$, $C_d^{*}(\uu) = \left(\PP(X(\0) \geq u_1), \ldots, \PP(X(\0) \geq u_m)\right)^t$ (see~\eqref{areadensity}) and $\widehat{C}_d^{(\delta,T)}(\uu)  = \left(\widehat{C}_d^{(\delta,T)}(u_1) , \ldots,    \widehat{C}_d^{(\delta,T)}(u_m)\right)^t$ (see Equation~\eqref{eq:C2pixel}). Then it holds that \end{sloppypar}
\begin{align}\label{TCLC2tilde}
\sigma_{d}(T)^{1/2} \, \big(\widehat{C}_d^{(\delta,T)}(\uu) - C_d^{*}(\uu)\big)  \stackrel{\mathcal{L}}\longrightarrow \mathcal{N}(0, \Sigma(\uu)),
\end{align}
as $N{\delta}\to\infty$ and $\delta\to0$, where $\Sigma(\uu) =\left(\upsilon_{ij}(\uu)\right)_{i, j = 1}^m$  is an ($m \times m$)-matrix having the elements
\begin{align}\label{CovarianceBulinski}
\upsilon_{ij}(\uu) = \int_{\R^d}   \Cov(\mathbbm{1}_{\{X(\0)\ge u_i\}},\mathbbm{1}_{\{X(t)\ge u_j\}})  \rmd t.\end{align}
\end{Corollary}

\begin{proof}[Proof of Corollary~\ref{CLT}]\begin{sloppypar}
Recall that  $\sigma_d(T_{N}) = (2 N \delta)^d$ and ${C}_d^{T}(\uu)   =  \widehat{C}_d^{ (\delta,T)}(\uu) +  R_\delta(\uu),$
where $R_{\delta}(\uu) := (R_{\delta}(u_1), \ldots, R_{\delta}(u_m))$
with\end{sloppypar}
$$R_{\delta}(u_k) := \frac{1}{\sigma_d(T_{N})} \sum_{\ii =-N}^{N-1} \int_{V_{\ii}(\delta)}\Big(\mathbbm{1}_{\{X(t)\ge u_k\}}-\mathbbm{1}_{\{X(t_{\ii, \delta})\ge u_k\}}\Big)\rmd t,\quad \mbox{for}\ \ k \in \{1, \ldots, m\}.$$
Then,
\begin{align*}
\sqrt{\sigma_d(T_{{N}})}\, \big(\widehat{C}_d^{(\delta,T)}(\uu) - C_d^{*}(\uu)\big)
 & = \sqrt{\sigma_d(T_{{N}})} \,({C}_d^{T}(\uu) - C_d^{*}(\uu)) - \sqrt{\sigma_d(T_{{N}})}  R_{\delta}(\uu)\\
&:= \textbf{I}_{1, N}(\uu) -   \textbf{I}_{2, N}(\uu).
\end{align*}
Notice that, since $N\delta_{N}\to \infty$,  $T_{{N}}$ is a VH-growing sequence (see Definition 6 in \cite{bulinski2012central}).
Then, using  Theorem  2   in \cite{bulinski2012central},  it holds that $\textbf{I}_{1, N}(\uu) \overset{\mathcal{L}}{\underset{N\rightarrow\infty}{\longrightarrow}} \mathcal{N}(0, \Sigma(\uu))$, with $\overset{\mathcal{L}}{\underset{ }{\longrightarrow}}$ the convergence in law and  where
 $\Sigma(\uu) =\left(\upsilon_{ij}(\uu)\right)_{i, j = 1}^m$  is an ($m \times m$)-matrix having the elements
$\upsilon_{ij}(\uu)$ as in~\eqref{CovarianceBulinski}.
As for $k \in \{1, \ldots, m\}$, $\E[R_\delta(u_k)]=0$, $\V(R_{\delta}(u_k))=o(\sigma_d(T_{{N}})^{-1})$,   from  Proposition~\ref{prop:AireRDimensionD}, then $\textbf{I}_{2, N}(\uu)  \overset{\PP}{\underset{N\rightarrow\infty}{\longrightarrow}} 0$. By  Slutsky's Theorem,   we get the result in~\eqref{TCLC2tilde}.
\end{proof}

\end{appendix}

\subsection{Examples}\label{sec:annexeExamples}

We begin with two simple examples that illustrate the use of the  Crofton formula in Equation~\eqref{eqn:crofton}.

\begin{Example}
Let $d=2$, Equation~\eqref{eqn:crofton} takes the form:
\begin{equation}\label{eqn:crofton_2d}
\sigma_1(M) = \frac{1}{4} \int_{-\infty}^\infty\int_0^{2\pi} \sigma_0(M \cap l_{\mathbf{s}(\theta),v})\ \rmd \theta\ \rmd v,
\end{equation}
where $\mathbf{s}(\theta):= (\cos(\theta),\sin(\theta))$. Let $M$ be a circle of radius $R$ in $\R^2$. For all $\theta\in[0,2\pi)$, the function $f_{\theta}(v):= \sigma_0(M \cap l_{\mathbf{s}(\theta),v})$ is equal to 2 on an interval of length $2R$, and 0 elsewhere. Therefore, we easily  recover
\begin{align*}
\sigma_1(M) &= \frac{1}{4} \int_0^{2\pi} \int_{-\infty}^\infty f_{\theta}(v)\ \rmd v\ \rmd \theta
= \frac{1}{4} \int_0^{2\pi} 4R\ \rmd \theta
= 2\pi R.
\end{align*}
\end{Example}

\begin{Example}
Let $M$ be the boundary of a a square $K$ with side-length $a$. As $\partial K = \bigcup_{i=1}^4 \mbox{side}_i$, using the additivity of~\eqref{eqn:crofton_2d}, it suffices to consider only a single side of the square. Without loss of generality, let us suppose that $\mbox{side}_1$ is horizontal. Then, the function $f_{\theta}(v):= \sigma_0(\mbox{side}_1 \cap l_{{\mathbf{s}(\theta)},v})$ is equal to 1 on an interval of length $a|\sin \theta|$, and 0 elsewhere. Thus, we easily   recover that
\begin{align*}
\sigma_1(\mbox{side}_1) &= \frac{1}{4} \int_0^{2\pi} \int_{-\infty}^\infty f_{\theta}(v)\ \rmd v\ \rmd \theta
= \frac{1}{4} \int_0^{2\pi} a|\sin\theta|\ \rmd \theta
= \frac{a}{2} \int_0^{\pi} \sin\theta\ \rmd \theta
= a.
\end{align*}
\end{Example}

We now present two classical examples where the density $C_{d-1}^{*}(u)$ in \eqref{defCast}
can be explicitly obtained.  The interested reader is referred to Exercises 6.2.c and 6.3 in \cite{AW09} and to~\cite{bierme2018lipschitz}.

\begin{Example}\label{ModelsGaussianType1}
Let   $X = \{X(t),  t \in \R^d\}$,   be an isotropic  Gaussian field,  with zero mean, unit variance and second  spectral moment $\lambda$ finite satisfying Assumption~\ref{A0}.   Then   we get
\begin{align*}
C_{d-1}^{*}(u) =  \sqrt{\frac{\lambda}{\pi}}e^{-\frac{u^{2}}2}\frac{\Gamma\left(\frac{d+1}{2}\right)}{\Gamma\left(\frac{d}{2}\right)}.
\end{align*}
\smallskip

Let $X^{(K)}=  \{X^{(K)}(t), t \in \R^d\}$ be   an isotropic   chi-square  field    satisfying Assumption~\ref{A0}  with  $K \in \N^+$ degree of freedom such that $X^{(K)}(t) = X_1(t)^2+ \ldots + X_K(t)^2$ where $X_1(t), \ldots, X_K(t)$  are $K$  \emph{i.i.d.}  stationary isotropic Gaussian random fields defined as above. Then one can get
\begin{align*}
C_{d-1}^{*}(u)= \sqrt{\lambda} \left(\frac u2\right)^{\frac{K-1}{2}}e^{-\frac u2} \frac{\Gamma\left(\frac{ d+1}2\right)}{\Gamma\left(\frac K2\right)\Gamma\left(\frac d2\right)}.
\end{align*}
 \end{Example}

\subsection{Convergence of the bias factor}\label{sec:convergence_of_bias}
In this appendix, we provide a short justification for Equation~\eqref{eqn:bias_converge}.
The Crofton formula in~\eqref{eqn:crofton} applied to a random manifold $M$ can be written in terms of the conditional  expectation,
\begin{equation}\label{eqn:conditional_crofton}\sigma_{d-1}(M) = \frac{\sqrt{\pi}\ \Gamma(\frac{d+1}{2})}{\Gamma(\frac{d}{2})}
\int_{\R^{d-1}}\int_{\partial B^d_1} \frac{\E\big[\sigma_0(M\cap l_{\mathbf{s},\mathbf{v_\mathbf{s}(u)}})\ |\ \sigma_{d-1}(M)\big]}{\sigma_{d-1}(\partial B^d_1)}\ \rmd \mathbf{s}\ \rmd\mathbf{u},\ \as.
\end{equation}
By slightly modifying the proof of Theorem~\ref{prp:Pcrossing_Esurfacearea} so as to leverage~\eqref{eqn:conditional_crofton}, one obtains
$$\frac{1}{q}\P\big(X(\0) \leq u < X(q\be_1)\ |\ C^T_{d-1}(u)\big) \overset{L^1}{\underset{q\rightarrow 0}{\longrightarrow}} \frac{C^T_{d-1}(u)}{\beta_d}.$$
Likewise, by slightly adjusting the proof of Theorem~\ref{thm:honeycomb} accordingly, one writes
$$\E\big[\widehat C^{(\delta \dot{\mathcal{H}},T)}_{d-1}(u)\ |\ C^T_{d-1}(u)\big]\overset{L^1}{\underset{\delta\rightarrow 0}{\longrightarrow}}\frac{2d}{\beta_d}C^T_{d-1}(u).$$
Equation~\eqref{eqn:bias_converge} follows.

\subsection{Alternative approaches}\label{DimensionalConstantSQUARE}

\subsubsection*{Further elements on the dimensional constant $\beta_{d}$}

Here, we provide an alternative justification for the value of the bias factor $\frac{2d}{\beta_d}$ in~\eqref{eqn:honeycomb_deterministic} for the simple case of the hypercubic lattice by adapting in any dimension $d$ the methodology proposed for instance by \citet{miller1999}.

Remark that the $(d-1)$-dimensional surface $L_X(u)$ is $C^1$, and so it can be approximated arbitrarily well by a union of $(d-1)$-dimensional hyperplanar surfaces. Then, the expectation of $\widehat C_{d-1}^{(\delta\dot{\mathcal{G}},T)}$  in~\eqref{eq:C1tildeNEW} is linear over these hyperpanar surfaces. Thus, it suffices to consider the bias in the estimate of the $\sigma_{d-1}$ measure of a single hyperplanar surface with orientation vector distributed uniformly on $\partial B_1^d$. It is not difficult to see that the contribution to $\widehat C_{d-1}^{(\delta\dot{\mathcal{G}},T)}$ of a hyperplanar surface with orientation vector $\textbf{r}\in\partial B_1^d$ is $||\textbf{r}||_1$ times its $\sigma_{d-1}$ measure. Thus, the expected bias factor should be the average value of $||\textbf{r}||_1$, when $\textbf{r}$ is distributed uniformly on $\partial B_1^d$. Indeed,
\begin{equation}\label{constantbetad}
\frac{1}{\sigma_{d-1}(\partial B_1^d)} \int_{\partial B_1^d} ||\textbf{r}||_1\, \sigma_{d-1}(\rmd \textbf{r}) = \frac{2d}{\beta_d}.
\end{equation}

\subsubsection*{Building another approximation for the $\sigma_{d-1}$ measure}

We discuss shortly an alternative approach for the estimation of the considered surface area.  However it is not well adapted to our framework since it is not  accessible from the sole knowledge of the field $X$ on the grid $\mathbb G(\delta, T)$.  The interested reader is also referred to \cite{bierme:hal-02793752} (see in particular Equation (4.5) and Figure 4.2) in the case $d=2$.\\

A first idea, following definition   of the $d-1$ dimensional Haussdorff measure  in~\eqref{deltacovering}, and the fact the   considered lattice is a grid, is to consider the hypercubic cover  of $T$ induced by the grid  $\mathbb G$  in~\eqref{G} (see Section~\ref{hypercubic:section}).  To this end,   define $J^{(\delta)}(u):=\left\{\ii \in \Z^d:  V_{\ii}(\delta) \subseteq T,\ V_{\ii}(\delta) \cap L^{T}_X(u)\ne \emptyset \right\}$  and  introduce
\begin{align}\label{eq:C1overline}
\check{C}^T_{d-1}(u) := \frac{a(\delta) }{\sigma_d(T)}  (\sqrt{d} \delta)^{d-1}
, \quad \mbox{ where } \quad a(\delta) := \mbox{Card}\left(J^{(\delta)}(u)\right).
\end{align}
Indeed, $L_X^T(u) \subset \bigcup_{\ii \in J^{(\delta)}(u)} V_{\ii}(\delta)$ and $\mbox{diam}(V_{\ii}(\delta)) = \sqrt{d}\delta$ for all $\ii \in J^{(\delta)}(u)$. So $\bigcup_{\ii  \in J^{(\delta)}(u)} V_{\ii}(\delta)$ is a $(\sqrt{d} \delta)$-cover of $L^T_X(u)$ in the sense of~\eqref{deltacovering}.
The random quantity $a(\delta)$  in~\eqref{eq:C1overline} represents the number of cells of the grid whose intersection with the level set $L_X^T(u)$ is non empty and it  can be rewritten as follows
 \begin{align*}
 a(\delta)&=\sum_{\ii = - N}^{N-1}\mathbbm{1}_{\big\{\underset{s\in V_{\ii}(\delta)}{\min}X({s}) \leq  u < \underset{s\in V_{\ii}(\delta)}{\max}X({s})\big\}}.
 \end{align*}

This object is not empirically accessible from the discretized excursion set as it requires the knowledge of the field on $T$. For these reasons,  this quantity is not a good candidate to estimate $C_{d-1}$ in our sparse information setting, \emph{i.e.}, when we have the   sole observation of $X$ on $\mathbb G(\delta, T)$.

\end{document}